\documentclass[11pt]{amsart}
\usepackage{amsmath}
\usepackage{amssymb}
\usepackage{comment}
\usepackage{color}

\usepackage[dvips]{graphics}
\definecolor{dred}{rgb}{0,0,0.6}
\newtheorem{thm}{Theorem}[section]
\newtheorem{cor}[thm]{Corollary}
\newtheorem{lem}[thm]{Lemma}
\newtheorem{prop}[thm]{Proposition}
\theoremstyle{definition}

\theoremstyle{remark}
\newtheorem{rem}[thm]{\bf{Remark}}
\numberwithin{equation}{section}

\newcommand{\beas}{\begin{eqnarray*}}
\newcommand{\eeas}{\end{eqnarray*}}
\newcommand{\bes} {\begin{equation*}}
\newcommand{\ees} {\end{equation*}}
\newcommand{\be} {\begin{equation}}
\newcommand{\ee} {\end{equation}}
\newcommand{\bea} {\begin{eqnarray}}
\newcommand{\eea} {\end{eqnarray}}

\newcommand{\txt} {\textmd}

\newcommand{\R}{\mathbb R}
\newcommand{\C}{\mathbb C}

\newcommand{\N}{\mathbb N}
\newcommand{\Z}{\mathbb Z}
\newcommand{\la}{\lambda}
\newcommand{\g}{\mathfrak{g}}

\begin{document}

\title[Sharp Adams type inequalities on noncompact symmetric spaces]{Sharp Adams type inequalities for the fractional Laplace-Beltrami operator on noncompact symmetric spaces}

\author{Mithun Bhowmik}

\address{Department of Mathematics, Indian Institute of Science, Bangalore-560012, India}
\email{mithunb@iisc.ac.in, mithunbhowmik123@gmail.com}



\begin{abstract}
We establish sharp Adams type inequalities on Sobolev spaces $W^{\alpha, n/\alpha}(X)$ of any fractional order $\alpha< n$ on Riemannian symmetric space $X$ of noncompact type with dimension $n$ and of arbitrary rank. We also establish sharp Hardy-Adams inequalities on the Sobolev spaces $W^{n/2, 2}(X)$. For the real hyperbolic spaces, such results were recently obtained by J. Li et al. (Trans. AMS, 2020).  We use Fourier analysis on the  symmetric spaces to obtain these results. 
\end{abstract}



\maketitle


\section{Introduction} 
The study and understanding of various kinds of sharp Sobolev inequalities which describe the embedding of Sobolev spaces into $L^p$ spaces or H\"older spaces have been a matter of intensive research. They play an important role in calculus of variations, differential
geometry, harmonic analysis, partial differential equations and other areas of modern mathematics. It is well-known that the Sobolev embedding holds for the case of compact Riemannian manifolds. To be precise, let $(M, g)$ be a compact Riemannian manifold then the Sobolev embedding states that the Sobolev space $W^{k,p}(M)$ is continuously embedded into $L^{p^\ast}(M)$ where $p^\ast = \frac{np}{n-kp}$ provided $1\leq p <
\frac{n}{k}$. However, when $M$ is a complete noncompact manifold then the Sobolev
embedding is a non-trivial issue. In fact, there exists a complete noncompact Riemannian manifold $M$ for which the Sobolev embedding $W^{k,p}(M) \hookrightarrow L^{p^\ast}(M)$ does not hold for any $p$ satisfying $kp < n$.
We refer to \cite{Hb} for a detailed discussion on the
topic.

When $M$ is compact and $p=n/k$, 
the space $W^{k,p}(M)$ is continuously
embedded into $L^q(M)$ for all $q < \infty$ but not for $q = \infty$. When $M$ is a
bounded domain in $\R^n$ with smooth boundary, Trudinger \cite{Tr} established in the borderline case that $W_0^{1,n}(\Omega) \subset L_{\phi_n}(\Omega)$, where $L_{\phi_n}(\Omega)$
is the Orlicz space associated with the Young function $\phi_n(t) = \exp\left(\beta|t|^{n/(n-1)}\right)-1$ for some $\beta > 0$. In 1971, Moser sharpened the Trudinger inequality in \cite{Mo} by finding the optimal $\beta$.
He showed that there exists a positive constant $C$ depending only on $n$ such that
\bes
\sup_{u\in C_c^\infty(\Omega), \int_{\Omega} |\nabla u|^n\leq 1} \int_{\Omega} e^{\beta |u(x)|^{n/(n-1)}}~dx\leq C |\Omega|,
\ees
holds for every $\beta\leq \beta_n= n[\omega_{n-1}]^{1/(n-1)}$, where $\Omega$ is a bounded domain in $\R^n$, $|\Omega|$ denotes the volume of $\Omega$ and  $\omega_{n-1}$ is the surface measure of the unit sphere in $\R^n$. Moreover, when $\beta> \beta_n$, the above supremum is infinite.

In 1988, D. Adams extended such an inequality  on finite  domain to higher order Sobolev spaces as follows
\begin{thm} \cite{Ad} \label{thm-Ad}
Let $\Omega$ be a domain in $\R^n$ with finite Lebesgue measure and $m$ be a positive integer less than $n$. There is a constant $c_0= c_0(n,m)$ such that for all $u\in C^{m}(\R^n)$ with support contained in $\Omega$ and $\|\nabla^m u\|_{n/m}\leq 1$, the following uniform inequality holds
\be \label{adam-est}
\frac{1}{|\Omega|}\int_{\Omega} \exp\left( \beta(n, m)|u(x)|^{n/(n-m)}\right)~dx\leq c_0,
\ee
where
\beas
\beta(n, m) &=&  \frac{n}{\omega_{n-1}}\left[\frac{\pi^{n/2} 2^m \Gamma\left((m+1)/2\right)}{\Gamma\left((n-m+1)/2\right)}\right]^{n/(n-m)}, \:\: m \txt{ is odd};\\
&=&  \frac{n}{\omega_{n-1}}\left[\frac{\pi^{n/2} 2^m \Gamma\left(m/2\right)}{\Gamma\left((n-m)/2\right)}\right]^{n/(n-m)}, \:\: m \txt{ is even}.
\eeas
Furthermore, the constant $\beta(n,m)$ in (\ref{adam-est}) is sharp in the sense that if $\beta(n,m)$ is replaced by any larger number, then the integral in (\ref{adam-est}) cannot be bounded uniformly by any constant. 
\end{thm}
For $m=1$, it recovers the Trudinger-Moser inequality. In \cite{Ad}, Theorem \ref{thm-Ad} was proved by representing functions by Riesz kernels and establishing inequalities for integral operators governed by these kernels. There have been many extensions and generalizations of this result to different settings. For instance, L. Fontana in \cite{F} obtained a sharp version of the inequality (\ref{adam-est}) on compact Riemannian manifolds. When $\Omega$ is a Euclidean ball, more refined results have been established in recent years. In dimension two, Wang and Ye \cite{WY} proved a Hardy-Trudinger-Moser inequality on a planar disk $\mathbb B^2$ by combining the Trudinger-Moser inequality with the Hardy inequality.
Their result is as follows
\begin{thm} \label{thm-d2}
Let $\mathbb B^2=\{z=x+iy: |z|= \sqrt{x^2+y^2} < 1\}$. There exists a constant $C>0$ such that for all $u\in C_c^\infty(\mathbb B^2)$ satisfying
\bes
\|u\|_{\mathcal H}=\int_{\mathbb B^2} |\nabla u(z)|^2~dxdy-\int_{\mathbb B^2}\frac{u(z)^2}{(1-|z|^2)^2}~dxdy \leq 1,
\ees
we have
\bes
\int_{\mathbb B^2} e^{4\pi u(z)^2}~dxdy< C< \infty.
\ees
\end{thm}
This is the borderline case of the first order Hardy-Sobolev-Maz'ya inequality on $\mathbb B^2$ for any $1 \leq p < \infty$ \cite{B1, Mz}:
\bes
\int_{\mathbb B^2} |\nabla u(x)|^2~dx-\int_{\mathbb B^2}\frac{u(x)^2}{(1-|x|^2)^2}~dx \geq C \left(\int_{\mathbb B^2} |u(x)|^p~dx \right)^{\frac{1}{p}}.
\ees
In \cite{LY1}, Lu and Yang give a rearrangement-free argument of the result of \cite{WY} and show using Riemann mapping theorem that the Hardy-Trudinger-Moser inequality holds for any bounded and convex domain in $\R^2$. 

Several variants of Trudinger-Moser and Adams type inequalities has been established in unbounded domains of Euclidean spaces (see for instance \cite{Co}). In \cite{LL}, N. Lam and G. Lu developed a new approach to establish these types of sharp inequalities in unbounded domains of Euclidean
spaces without using the standard symmetrization. This approach can be applied in the context of Riemannian and sub-Riemannian manifolds where the symmetrization argument does not work (see e.g. \cite{LL2, YSK}). They proved the following Adams type inequalities on Sobolev space $W^{\alpha, n/\alpha}(\R^n)$ of fractional order $\alpha$ for  $0 < \alpha < n$:
\begin{thm} \label{thm-frac-rn}
Let $0< \alpha< n$ be an arbitrary real positive number, $p= n/\alpha$ and $\tau>0$. There holds
\bes
\sup_{u\in W^{\alpha, p}(\R^n), \|(\tau I-\Delta)^{\alpha/2}u\|_p\leq 1}\int_{\R^n} \Phi_p\left(\beta_0(n, \alpha)|u(x)|^{p^\prime}\right)~dx\leq C< \infty,
\ees
where $\beta_0(n, \alpha)= \frac{n}{\omega_{n-1}}\left[\frac{\pi^{n/2} 2^\alpha \Gamma\left(\alpha/2\right)}{\Gamma\left((n-\alpha)/2\right)}\right]^{p^\prime}$, $1/p+1/p^\prime=1$ and 
\bes
\Phi_p(t)= e^t- \sum_{j=0}^{j_p-2}\frac{t^j}{j!}, \:\: j_p= \min\{j\in \N: j\geq p\}.
\ees
Furthermore, this inequality is sharp in the sense that if $\beta_0(n, \alpha)$ is replaced by any $\beta> \beta_0(n, \alpha)$, then the supremum is infinite.
\end{thm}

In the case of real hyperbolic spaces $\mathbb B^n$, Trudinger-Moser, Adams inequalities have been investigated in details. From a conformal point of view, an Adams
inequality in the hyperbolic space was proved by Karmakar and Sandeep \cite{KS}. On the other hand, in a series of papers \cite{LY2, LLY1, LLY2}, using the Riesz kernel estimates and Fourier analysis on $\mathbb B^n$,  Li et al. proved sharp Adams and Hardy-Adams inequalities on $\mathbb B^n$. Precisely, in \cite{LLY2} they proved the results for any fractional order $\alpha< n$ on $\mathbb B^n$ in all dimension $n$. These generalize their earlier results for the sharp Hardy-Adams inequalities corresponding to the bi-Laplace-Beltrami operator $(-\Delta_{\mathbb B^4})^2$ on the hyperbolic space $\mathbb B^4$ of dimension four in \cite{LY2} and to $n/2$-th (integer) power of the Laplace-Beltrami operator $(-\Delta_{\mathbb B^n})^{n/2}$ on $\mathbb B^n$ of any even dimension $n \geq 4$  in \cite{LLY1}. The main result in \cite[Theorem 1.11]{LLY2} is as follows.
\begin{thm} \label{thm-h}
Let $n \geq 3, 0< \alpha< n$ be an arbitrary positive number, $p= n/\alpha$ and $\zeta$ satisfies $\zeta> (1/p- 1/2)(n-1)/2$  if $1< p< 2$ and $\zeta> (1/2-1/p)(n-1)$ if $p\geq 2$. Then there exists $C=C(\zeta, n, \alpha)$ such that
\bes
\int_{\mathbb B^n} \Phi_p\left(\beta_0(n, \alpha)|u(x)|^{p^\prime}\right)~dV(x)\leq C,
\ees
for any $u\in W^{\alpha, p}(\mathbb B^n)$ with $\int_{\mathbb B^n} |\left(-\Delta_{\mathbb B^n}-(n-1)^2/4+\zeta^2\right)^{\alpha/2}u(x)|^p~dV(x)\leq 1$. Here $\beta_0$ and $\Phi_p$ are defined as in Theorem \ref{thm-frac-rn} and $dV(x)= 2^n(1-|x|)^{-2n}~ dx$ is the hyperbolic volume.
Furthermore, this inequality is sharp in the sense that if $\beta_0(n, \alpha)$ is replaced by any $\beta> \beta_0(n, \alpha)$, then the inequality can no longer hold with some C independent of $u$.
\end{thm}
We notice that $(n-1)|1/2-1/p|< (n-1)/2$ provided $p>1$. Choosing $\zeta= (n-1)/2$ in Theorem \ref{thm-h}, one get the following Adams inequality on $\mathbb B^n$ \cite[Theorem 1.12]{LLY2}.
\begin{thm} \label{thm2}
Let $n\geq 3, 0< \alpha< n$ be an arbitrary real positive number and $p= n/\alpha$. Then there exists $C=C(n, \alpha)$ such that
\bes
\int_{\mathbb B^n} \Phi_p\left(\beta_0(n, \alpha)|u(x)|^{p^\prime}\right)~dV(x)\leq C,
\ees
for any $u\in W^{\alpha, p}(\mathbb B^n)$ with $\int_{\mathbb B^n} |(-\Delta_{\mathbb B^n})^{\alpha/2} u(x)|^p~dV(x)\leq 1$.
\end{thm}
To prove Theorem \ref{thm-h}, the authors derived the optimal bounds for the Green functions of the fractional
Laplacians when the hyperbolic distance is small. To get these optimal bounds, they heavily used the explicit expression of the heat kernel available for the hyperbolic spaces. For the bounds corresponding to large hyperbolic distance, they used the results due to J. Anker and L. Ji \cite{AJ}.
In the same paper \cite{LLY2}, Li et al. have also established fractional order Sobolev embedding theorem on hyperbolic spaces and using this they proved the following Hardy-Adams inequality on $\mathbb B^n$ for $p=2$.
\begin{thm}\cite[Theorem 1.14]{LLY2} \label{thm-LLY2}
Let $n\geq 3, \zeta>0$ and $s$ satisfies $1\leq s< 3/2$ if  $n\geq 6$, $1\leq s\leq 5/4$ if $n=5$ and $s=1$ if $n=3, 4$. Then there exists $C=C(\zeta, n)$  such that
\bes
\int_{\mathbb B^n} \left[e^{\beta_0({n,n/2})u(x)^2}-1- \beta_0(n, n/2)u(x)^2\right]~dV(x)\leq C,
\ees
for any $u\in W^{n/2, 2}(\mathbb B^n)$ with
\be \label{eqn-ha}
\int_{\mathbb B^n} |(-\Delta_{\mathbb B^n})^{s/2}(-\Delta_{\mathbb B^n}-\frac{(n-1)^2}{4}+\zeta^2)^{(n-2s)/4} u(x)|^2~dV(x) - C_{\zeta,s} \int_{\mathbb B^n}|u(x)|^2~dV(x)\leq 1,
\ee
where $C_{\zeta,s}=\frac{(n-1)^{2s}\zeta^{n-2s}}{4^s}$.
\end{thm}
The result above has been obtained under  a weaker assumption, in particular, replacing (\ref{eqn-ha}) by 
\bes
\int_{\mathbb B^n} |(-\Delta_{\mathbb B^n}-(n-1)^2/4)^{s/2}(-\Delta_{\mathbb B^n}-(n-1)^2/4+\zeta^2)^{(n-2s)/4} u(x)|^2~dV(x)\leq 1.
\ees
In the case of even dimension $n > 4$, using Fourier analysis on hyperbolic spaces, it was shown in \cite{LLY1} that for any $\zeta>0$
\beas
&& \int_{\mathbb B^n} (-\Delta_{\mathbb B^n}-(n-1)^2/4)(-\Delta_{\mathbb B^n}-(n-1)^2/4+\zeta^2)^{n/2-1} u(x) \cdot~u(x)~dV(x)\\
&\leq& \int_{\mathbb B^n} |\nabla^{\frac{n}{2}} u(x)|^2~dx -\prod_{k=1}^{n/2} (2k-1)^2 \int_{\mathbb B^n} \frac{u(x)^2}{(1-|x|^2)^n}~dx, \:\: \textit{ for } u\in C_c^\infty(\mathbb B^n).
\eeas
Using the above inequality the following Hardy-Adams inequalities was obtained in \cite[Theorem 1.8]{LLY1} for all hyperbolic spaces $\mathbb B^n$ of even dimension $n\geq 4$.
\begin{thm} 
There exists a constant $C>0$ such that for all $u\in C_c^\infty(\mathbb B^n)$ with
\bes
\int_{\mathbb B^n} |\nabla^{\frac{n}{2}} u(x)|^2~dx -\prod_{k=1}^{n/2} (2k-1)^2 \int_{\mathbb B^n} \frac{u(x)^2}{(1-|x|^2)^n}~dx \leq 1,
\ees
there holds
\bes
\int_{\mathbb B^n} e^{\beta_0(n, n/2) u(x)^2}-1-\beta_0(n, n/2)u(x)^2~dV(x) \leq C.
\ees
\end{thm}
This is a borderline case of the sharp higher order Hardy-Sobolev-Maz'ya inequalities on half-spaces $\R^n_+$ and hyperbolic spaces $\mathbb B^n$ proved in \cite{Lu-Yang}.

Very recently, Bertrand and Sandeep \cite{BS} established a Trudinger-Moser-Adams inequality on
Cartan-Hadamard manifold with strictly negative sectional curvature. For other Trudinger-Moser-Adams inequalities on Riemannian manifolds, we
refer to \cite{K, S, SY, Y2, YSK}. 

In the case of real hyperbolic space $\mathbb B^n$, one can carry out very explicit calculations, since Fourier analysis and harmonic analysis tools are available. The space $\mathbb B^n$ is the
simplest example of a Riemannian symmetric space of rank one. Our concern in this article is to establish these inequalities on Riemannian symmetric spaces $X$ of noncompact type of all dimension $n\geq 3$ and of arbitrary rank. Precisely, we prove sharp local and global Adams inequalities (Theorem \ref{thm0}, Theorem \ref{thm1})   on the fractional order Sobolev spaces $W^{\alpha, n/ \alpha}(X)$, $0< \alpha< n$ and Hardy-Adams inequalities (Theorem \ref{thm-HA}) on $W^{n/2, 2}(X)$. Let $\Delta$ denote the Laplace-Beltrami operator on $X$ and let $\rho$ denote the half-sum of all positive roots counted with their multiplicities (see (\ref{defn-rho}) for the definition). We begin with the following sharp local and global Adams inequalities of fractional order on $X$.

\begin{thm} \label{thm0}
Let $n\geq 3, 0< \alpha<n$ be an arbitrary positive number, $p= n/\alpha$ and $\zeta$ satisfies $\zeta> 0$ if $1< p<2$ and $\zeta> 2|\rho|(1/2-1/p)$ if $p\geq 2$. Then for a measurable set $E$ with finite volume in $X$, there exists $C=C(\zeta, n, \alpha, |E|)$ such that
\bes
\frac{1}{|E|} \int_{E} \exp\left(\beta_0(n, \alpha)|u(x)|^{p^\prime}\right)~dx \leq C,
\ees
for any $u\in W^{\alpha, p}(X)$ with $\int_X |(-\Delta- |\rho|^2+\zeta^2)^{\alpha/2}u(x)|^p~dx \leq 1$. Furthermore, this inequality is sharp in the sense that if $\beta_0(n, \alpha)$ is replaced by any $\beta> \beta_0(n, \alpha)$, then the inequality can no longer hold with some C independent of $u$.
\end{thm}

\begin{thm} \label{thm1}
Let $n\geq 3, 0< \alpha< n$ be an arbitrary positive number, $p= n/\alpha$ and $\zeta$ satisfies $\zeta> 2|\rho||1/2-1/p|$. Then there exists $C=C(\zeta, n, \alpha)$ such that
\bes
\int_X \Phi_p\left(\beta_0(n, \alpha)|u(x)|^{p^\prime}\right)~dx\leq C,
\ees
for any $u\in W^{\alpha, p}(X)$ with $\int_X |(-\Delta- |\rho|^2 +\zeta^2)^{\alpha/2} u(x)|^p~dx\leq 1$.
Furthermore, this inequality is sharp in the sense that stated in Theorem \ref{thm0}.  
\end{thm}

We notice that $2|\rho||1/2-1/p|< |\rho|$ provided $p>1$. Choosing $\zeta= |\rho|$ in Theorem \ref{thm1}, we have the following Adams inequality.
\begin{thm} \label{thm2}
Let $n\geq 3, 0< \alpha< n$ be an arbitrary positive number and $p= n/\alpha$. Then there exists $C=C(\zeta, n, \alpha)$ such that
\bes
\int_X \Phi_p\left(\beta_0(n, \alpha)|u(x)|^{p^\prime}\right)~dx\leq C,
\ees
for any $u\in W^{\alpha, p}(X)$ with $\int_X |(-\Delta)^{\alpha/2} u(x)|^p~dx\leq 1$.
\end{thm}

In \cite{BP}, the author in collaboration with S. Pusti have established a fractional Poincar\'e-Sobolev inequality on $X$ (see Theorem \ref{thm-p-s}). Using this we prove the following result in the special case $p=2$. 
\begin{thm} \label{thm3}
Let $n\geq 3, \zeta>0$ and $s$ satisfies $0< 2s< \min\{l+2|\Sigma_0^+|, n\}$. Then there exists $C=C(\zeta, n)$  such that
\bes
\int_{X} \left[\exp\left(\beta_0({n,n/2})|u(x)|^2\right)-1- \beta_0(n, n/2)|u(x)|^2\right]~dx\leq C,
\ees
for any $u\in W^{n/2, 2}(X)$ with
\be \label{hypo-thm3}
\int_X \big|\left(-\Delta-|\rho|^2\right)^{s/2}\left(-\Delta-|\rho|^2+\zeta^2\right)^{(n-2s)/4} u(x)\big|^2~dx\leq 1.
\ee
Furthermore, this inequality is sharp in the sense stated above.
\end{thm}
\begin{rem}
The number $\nu=l+2|\Sigma_0^+|$ is called `pseudo-dimension' (see (\ref{defn-pseu}) for the definition). In the case of rank one symmetric spaces, in particular, for real hyperbolic space $\mathbb B^n$ of dimension $n$, the pseudo-dimension $\nu=3$. 
\end{rem}
Theorem \ref{thm3} implies the following Hardy-Adams inequality.
\begin{thm} \label{thm-HA}
Let $n\geq 3, \zeta>0$ and $s$ satisfies $2\leq 2s< \min\{l+2|\Sigma_0^+|, n\}$. Then there exists $C=C(\zeta, n)$  such that
\bes
\int_{X} \left[\exp\left(\beta_0({n,n/2})|u(x)|^2\right)-1- \beta_0(n, n/2)|u(x)|^2\right]~dx\leq C,
\ees
for any $u\in W^{n/2, 2}(X)$ with
\bes
\int_X \big|(-\Delta)^{s/2}(-\Delta-|\rho|^2+\zeta^2)^{(n-2s)/4} u(x)\big|^2~dx- |\rho|^{2s}\zeta^{n-2s}\int_{X}|u(x)|^2~dx\leq 1.
\ees
\end{thm}
\begin{rem}
In contrast with Theorem \ref{thm-LLY2} on the real hyperbolic spaces, our result improves the range of $s$ with $1\leq s< 3/2$ for all dimension $n$. 
\end{rem}
The article is organized as follows. In section 2, we review some preliminaries of Riemannian symmetric spaces and Fourier analysis on them. Using Anker's multiplier theorem on $X$, we derive a Sobolev embedding theorem on fractional Sobolev spaces (Corollary \ref{cor-sobolev}). Section 3 focuses on the optimal Bessel-Green-Riesz kernel estimates near the origin for the fractional operators. We also need to establish sharp estimates for the convolution of the fractional kernels. Section 4 devotes to the preparation of the proof of the important local Adams inequality (Theorem \ref{thm0}). In section 5, we prove all the results using Fourier analysis on $X$.

\section{Riemannian symmetric spaces of noncompact type}
In this section, we describe the necessary preliminaries regarding semisimple Lie groups and harmonic analysis on Riemannian symmetric spaces. These are standard and can be found, for example, in \cite{GV, H, H1, H2}. To make the article self-contained, we shall gather only those results which will be used throughout this paper. 

\subsection{Notations} Let $G$ be a connected, noncompact, real semisimple Lie group with finite center and $\mathfrak g$ its Lie algebra. We fix a Cartan involution $\theta$ of $\mathfrak g$ and write $\mathfrak g = \mathfrak k \oplus \mathfrak p$ where $\mathfrak k$ and $\mathfrak p$ are $+1$ and $-1$ eigenspaces of $\theta$ respectively. Then $\mathfrak k$ is a maximal compact subalgebra of $\mathfrak g$ and $\mathfrak p$ is a linear subspace of $\mathfrak g$. The Cartan involution $\theta$ induces an automorphism $\Theta$ of the group $G$ and $K=\{g\in G\mid \Theta (g)=g\}$ is a maximal compact subgroup of $G$.  Let $B$ denote the Cartan Killing form of $\mathfrak g$. It is known that $B\mid_{\mathfrak p\times\mathfrak p}$ is positive definite and hence induces an inner product and a norm $| \cdot |$ on $\mathfrak p$. The homogeneous space
X = G/K is a Riemannian symmetric space of noncompact type. The tangent space of $X$ at the point $o=eK$ can be naturally identified to $\mathfrak p$ and the restriction of $B$ on $\mathfrak p$ then induces a $G$-invariant Riemannian metric ${\bf d}$ on $X$. For $x\in X$, we denote $|x|$ is the distance of $x$ from the origin $o$. For a given $x\in X$ and a positive number $r$ we define
\bes
{\bf B}(x, r)=\{y\in X : \:\:{{\bf d}}(x, y)<r\},
\ees
to be the open ball with center $x$ and radius $r$.

We fix a maximal abelian subspace $\mathfrak a$ in $\mathfrak p$. The rank of $X$ is the dimension $l$ of $\mathfrak a$.
We shall identify $\mathfrak a$ endowed with the inner product induced from $\mathfrak p$ with $\mathbb{R}^l$  and let $\mathfrak{a}^*$ be the real dual of $\mathfrak{a}$. The set of restricted roots of the pair $(\g, \mathfrak{a})$ is denoted by $\Sigma$.  It consists of all $\alpha \in \mathfrak{a}^*$ such that
\bes
\g_\alpha = \left\{X\in \g ~|~ [Y, X] = \alpha(Y) X, \:\: \txt{ for all } Y\in \mathfrak{a} \right\},
\ees
is non-zero with $m_\alpha = \dim(\g_\alpha)$. We choose a system of positive roots $\Sigma_+$ and with respect to $\Sigma_+$, the positive Weyl chamber
$\mathfrak{a}_+ = \left\{X\in \mathfrak{a} ~|~ \alpha(X)>0,\:\:  \txt{ for all } \alpha \in \Sigma_+\right\}$. We also let $\Sigma_0^+$ be the set of positive indivisible roots, that is, $\Sigma_0^+=\{\alpha\in \Sigma^+\mid 2\alpha\not\in \Sigma\}$. Let $n$ be the dimension of $X$ and $\nu$ be the pseudo-dimension:  
\be \label{defn-pseu}
n=l+ \sum_{\alpha \in \Sigma_+}m_\alpha,\:\: \textit{ and } \nu=l+ 2|\Sigma_0^+|.
\ee
We notice that one cannot
compare $n$ and $\nu$ without specifying the geometric structure of $X$. For example, when $G$ is complex, we have $n=\nu$; but when $X$ has normal real form, we have $n = l + |\Sigma_0^+|$ which is strictly smaller than $\nu$. Let $\rho \in \mathfrak a^\ast$ denote the half-sum of all positive roots counted with their multiplicities
\be \label{defn-rho}
\rho=\frac{1}{2}\sum_{\alpha\in \Sigma_+}m_{\alpha}\alpha.
\ee 
It is known that the $L^2$-spectrum
of the Laplace-Beltrami operator $\Delta$ on $X$ is the half-line $(-\infty, -|\rho|^2]$.
Let $\mathfrak n$ be the nilpotent Lie subalgebra of $\mathfrak g$ associated to $\Sigma_+$, that is, $\mathfrak{n}= \oplus_{\alpha \in \Sigma_+}  ~ \mathfrak{g}_{\alpha}$. If $N=\exp \mathfrak{n}$ and $A= \exp \mathfrak{a}$ then $N$ is a nilpotent Lie subgroup and $A$ normalizes $N$. For the group $G$, we now have the Iwasawa decomposition 
$G= KAN$, that is, every $g\in G$ can be uniquely written as 
\bes
g=\kappa(g)\exp H(g)\eta(g), \:\:\:\: \kappa(g)\in K, H(g)\in \mathfrak{a}, \eta(g)\in N,
\ees 
and the map $(k, a, n) \mapsto kan$ 
is a global diffeomorphism of $K\times A \times N$ onto $G$.
Let $M'$ and $M$ be the normalizer and centralizer of $\mathfrak{a}$ in $K$ respectively.
Then $M$ is a normal subgroup of $M'$ and normalizes $N$. The quotient $W = M'/M$ is a finite group, called the Weyl group of the pair $(\g, \mathfrak{k})$. $W$ acts on $\mathfrak{a}$ by the adjoint action. It is known that $W$ acts as a group of orthogonal transformations (preserving the Cartan-Killing form) on $\mathfrak{a}$. Each $w\in W$ permutes the Weyl chambers and the action of $W$ on the Weyl chambers is simply transitive. Let $A_+= \exp{\mathfrak{a_+}}$. Since $\exp: \mathfrak{a} \to A$ is an isomorphism we can identify $A$ with $\R^l$. Let $\overline{A_+}$ denote the closure of $A_+$ in $G$. One has the polar decomposition $G=K A K$,
that is, each $g\in G$ can be written as 
\bes
g=k_1 (\exp Y) k_2, \:\:  k_1, k_2 \in K, Y\in \mathfrak{a}.
\ees 
In the above decomposition, the $A$ component of $\mathfrak{g}$ is uniquely determined modulo $W$. In particular, it is well defined in $\overline{A_+}$. 
The map $(k_1, a, k_2)\mapsto k_1ak_2$ of $K\times A\times K$ into $G$ induces a diffeomorphism of $K/M\times A_+\times K$ onto an open dense subset of $G$. 
It follows that if $gK=k_1 (\exp Y)K\in X$ then
\bes
|gK|={\bf d}(o, gK)=|Y|.
\ees
We recall the following property of the Iwasawa projection map $H$ \cite[Lemma 1.14, p.217]{H1}:
\be \label{Iwasawa}
|H(\exp Y k)|\leq |Y|, \:\:\:\: \textit{ for } Y\in \mathfrak a, k\in K.
\ee
We extend the inner product on $\mathfrak{a}$ induced by $B$ to $\mathfrak{a}^*$ by duality, that is, set
\bes
\langle \la, \mu \rangle =B(Y_\la, Y_\mu), \:\:\:\: \la, \mu \in \mathfrak{a}^*,  ~ Y_\la, Y_\mu \in \mathfrak{a},
\ees
where $Y_\la$ is the unique element in $\mathfrak{a}$ such that 
\bes
\la(Y) = B(Y_\la, Y), \:\:\:\: \txt{ for all } Y\in \mathfrak{a}.
\ees
This inner product induces a norm, again denoted by $|\cdot|$, on $\mathfrak{a}^*$,
\bes
|\la| = \langle \la, \la \rangle^{\frac{1}{2}}, \:\:\:\: \la \in \mathfrak{a}^*.
\ees
The elements of the Weyl group $W$ act on $\mathfrak a^*$ by the formula
\bes
sY_{\la}=Y_{s\la},\:\:\:\:\:\:s\in W,\:\la\in\mathfrak a^*.
\ees
Let $\mathfrak{a}_\C^*$ denote the complexification of $\mathfrak{a}^*$, that is, the set of all complex-valued real linear functionals on $\mathfrak{a}$. The inner products have complex bilinear extensions to the complexifications $\mathfrak a_\C$ and
$\mathfrak a_\C^\ast$. All these bilinear forms are denoted by the same symbol $\langle \cdot, \cdot \rangle$.

Through the identification of $A$ with $\R^l$, we use the Lebesgue measure on $\R^l$ as the Haar measure $da$ on $A$. As usual on the compact group $K$, we fix the normalized Haar measure $dk$ and $dn$ denotes a Haar measure on $N$. The following integral formulae describe the Haar measure of $G$ corresponding to the Iwasawa and polar decompositions respectively.
\beas
\int_{G}{f(g)dg} &=& \int_K \int_{\mathfrak{a}}\int_N f(k\exp Y n)~e^{2\rho(Y)}\:dn\:dY\:dk,\:\:\:\:\:\: f\in C_c(G); \\ 
&=&\int_{K}{\int_{\overline{A_+}}{\int_{K}{f(k_1ak_2) ~ J(a)\:dk_1\:da\:dk_2}}},
\eeas
where $dY$ is the Lebesgue measure on $\R^l$ and for $Y\in \overline{\mathfrak{a}_+}$
\be \label{J-est}
J(\exp Y)= c \prod_{\alpha\in \Sigma^+}\left(\sinh\alpha(Y)\right)^{m_{\alpha}}  \asymp \left\{\prod_{\alpha\in \Sigma^+}\left(\frac{\alpha(Y)}{1 + \alpha(Y)}\right)^{m_\alpha} \right\} e^{2\rho(Y)},
\ee
where $c$ (in the equality above) is a normalizing constant. 
If $f$ is a function on $X= G/K$ then $f$ can be thought of as a function on $G$ which is right invariant under the action of $K$. It follows that on $X$ we have a $G$ invariant measure $dx$ such that 
\be \label{int-formula}
\int_X f(x)~dx= \int_{K/M}\int_{\mathfrak{a}_+}f(k\exp Y)~J(\exp Y)~dY~dk_M,
\ee
where $dk_M$ is the $K$-invariant measure on $K/M$.  

\subsection{Fourier analysis on $X$}
For a sufficiently nice function $f$ on $X$, its Fourier transform $\widetilde{f}$ is defined on $\mathfrak{a}_{\C}^* \times K$ by the formula 
\be \label{hftdefn}
\widetilde{f}(\la,k) = \int_{G} f(g) e^{(i\la - \rho)H(g^{-1}k)} dg,\:\:\:\:\:\: \la \in \mathfrak{a}_{\C}^*,\:\: k \in K, 
\ee
whenever the integral exists \cite[P. 199]{H1}. 
As $M$ normalizes $N$ the function $k\mapsto\widetilde{f}(\la, k)$ is right $M$-invariant.
It is known that if $f\in L^1(X)$ then $\widetilde{f}(\la, k)$ is a continuous function of $\la \in \mathfrak{a}^*$, for almost every $k\in K$ (in fact, holomorphic in $\la$ on a domain containing $\frak a^*$). If in addition, $\widetilde{f}\in L^1(\mathfrak{a}^*\times K, |{\bf c}(\la)|^{-2}~d\la~dk)$ then the following Fourier inversion holds,
\be\label{hft}
f(gK)= |W|^{-1}\int_{\mathfrak{a}^*\times K}\widetilde{f}(\la, k)~e^{-(i\la+\rho)H(g^{-1}k)} ~ |{\bf c}(\la)|^{-2}d\la~dk,
\ee
for almost every $gK\in X$ \cite[Chapter III, Theorem 1.8, Theorem 1.9]{H1}. Here ${\bf c}(\la)$ denotes Harish Chandra's $c$-function. Moreover, $f \mapsto \widetilde{f}$ extends to an isometry of $L^2(X)$ onto $L^2(\mathfrak{a}^*_+\times K, |{\bf c}(\la)|^{-2}~d\la~dk )$ \cite[Chapter III, Theorem 1.5]{H1},
that is, 
\be \label{plancherel}
\int_X|f(x)|^2 dx = c_G \int_{\mathfrak{a}^*_+\times K} |\tilde{f}(\lambda, k)|^2~ |{\bf c}(\lambda)|^{-2}~d\lambda~dk,
\ee 
where $c_G$ is a positive number that depends only on $G$.

We now specialize to the case of $K$-biinvariant functions $f$, that is, $f$ satisfies $f(k_1 gk_2)= f(g)$, for all $k_1, k_2\in K$ and $g\in G$. Using the polar decomposition of $G$ we may view an integrable or a continuous $K$-biinvariant function $f$ on $G$ as a function on $A_+$, or by using the inverse exponential map we may also view $f$ as a function on $\mathfrak{a}$ solely determined by its values on $\mathfrak{a}_+$. Henceforth, we shall denote the set of $K$-biinvariant functions in $L^p(G)$ by $L^p(G//K)$, for $1\leq p\leq \infty$; and $K$-biinvariant compactly supported smooth functions by $C_c^\infty(G//K)$.
If $f\in L^1(G//K)$ then the Fourier transform $\widetilde{f}$ can also be written as
\be \label{hlsphreln}
\widetilde f(\la, k) =\widehat{f}(\la )= \int_Gf(g)\phi_{-\la}(g)~dg,
\ee
where $\phi_\lambda$ is Harish Chandra's elementary spherical function defined by
\be \label{philambda} 
\phi_\la(g) 
= \int_K e^{-(i\la+ \rho) \big(H(g^{-1}k)\big)}~dk,\:\:\:\:\:\:\la \in \mathfrak{a}_\C^*.  
\ee
We now list down some well-known properties of the elementary spherical functions which are important for us (\cite[Prop 3.1.4]{GV}, \cite[Prop. 2.2.12]{AJ}, \cite[Thm 1.1, p. 200; Lemma 1.18, p. 221]{H1}).
\begin{thm} \label{philambdathm}
\begin{enumerate}
\item[(1)] $\phi_\la(g)$ is $K$-biinvariant in $g\in G$ and $W$-invariant in $\la\in \mathfrak{a}_\C^*$.
\item[(2)] $\phi_\la(g)$ is $C^\infty$ in $g\in G$ and holomorphic in $\la\in \mathfrak{a}_\C^*$.
\item[(3)] For all $\la\in \overline{\mathfrak{a}_+^*}$ and $g\in G$ we have $|\phi_\la(g)| \leq  \phi_0(g)\leq 1$.
\item[(4)] The elementary spherical function $\phi_0$ satisfies the following estimate:
\be \label{est-phi0}
 \phi_0(\exp Y)\asymp \left\{\prod_{\alpha\in \Sigma_0^+} \left(1 + \alpha(Y)\right) \right\}  e^{-\rho(Y)}, \:\:\text{ for all } Y\in \overline{\mathfrak a^+}.
\ee
\item[(5)]For $\la\in \mathfrak{a}^*$, there holds $\Delta \phi_{\la}=-(|\la|^2+|\rho|^2)\phi_{\la}$.

\item[6)] For $\la\in \mathfrak a^*$, the function $\phi_\la$ satisfies the following
\be \label{phi-xy}
\phi_{-\la}(hg) = \int_K e^{(i\la - \rho)\big(H(g^{-1}k) \big)}e^{-(i\la+\rho)\big(H(hk) \big)}~dk, \:\:\:\: g, h \in G.
\ee		
\end{enumerate}
\end{thm}
We now recall the following asymptotic estimates of the heat kernel $h_t(x)$ on $X$ established by Anker and Ji \cite[Theorem 3.7]{AJ}.
\begin{thm}\label{est-heatkernel-bothside}
Let $\kappa$ be an arbitrary positive number. Then there exist positive constants $C_1, C_2$ $($depending on $\kappa$) such that 
\bes
C_1\leq \frac{h_t(\exp Y)}{t^{-\frac{n}{2}}(1+t)^{\frac{n-l}{2}-|\Sigma_0^{+}|} \left\{\prod_{\alpha\in \Sigma_0^{+}}(1+\alpha(Y)\right\}e^{-|\rho|^2t-\rho(Y)- \frac{|Y|^2}{4t}}} \leq C_2,
\ees
for all $t>0$, and $Y\in \overline{\mathfrak a^+}$, with $|Y|\leq \kappa (1+t)$.
\end{thm}
Let $\alpha>0$ and $1 < p < \infty$. We recall that the Sobolev space $W^{\alpha, p}(X)$ is the image of $L^p(X)$ under the operator $(-\Delta)^{-\alpha/2}$, equipped with the norm
\bes
\|f\|_{W^{\alpha, p}(X)}= \|(-\Delta)^{\alpha/2} f\|_{L^p(X)}.
\ees 
If $\alpha = N$ is a non-negative integer, then $W^{\alpha, p}(X)$ coincides with the classical Sobolev
space
\bes
W^{N, p}(X) = \{f \in L^p(X) : \nabla^j f \in L^p(X), \forall 1 \leq j \leq N \},
\ees
defined by means of covariant derivatives. We refer to \cite{Tb} for more details about
function spaces on Riemannian manifolds.
For $p=2$, the  Sobolev space of order $\alpha$ on $X$ is equivalently defined by  
\bes
W^{\alpha, 2}(X)= \big\{f\in L^2(X) ~|~ \|f\|_{W^{\alpha, 2}(X)}^2:= \int_{\mathfrak a^\ast \times K} |\tilde f(\lambda, k)|^2 ~(|\lambda|^2+|\rho|^2)^{\sigma}~|{\bf c}(\lambda)|^{-2}~d\lambda~dk< \infty \big\}.
\ees 
In \cite{An1}, J.-P. Anker proved the following H\"ormander-Mikhlin type multiplier theorem in the context of Riemannian symmetric spaces of the noncompact type. Let $k$ be a $K$-biinvariant tempered distribution on $G$ and let $m$ be its spherical Fourier transform.
\begin{thm}
Let $1 < p < \infty, ~v = |1/p - 1/2|$ and $N = [v n]+ 1$. Then $Tf = f \ast k$ is a bounded operator on $L^p(X)$, provided that
\begin{enumerate}
\item[(a)] $m$ extends to a holomorphic function inside the tube $\mathcal I^v= \mathfrak a^\ast + i~ co(W.2v \rho)$,

\item[(b)] $\nabla^i m$ (for $i=0, \cdots, N)$  extends continuously to the whole of $\mathcal I^v$, with
\bes
\sup_{\lambda \in \mathcal I^v}(1+|\lambda|)^{-i}~|\nabla^i m(\lambda)|< \infty.
\ees
\end{enumerate}
\end{thm}
The multiplier of the operator $(-\Delta-|\rho|^2 + \zeta^2)^{-\alpha/2}$ is given by $(\langle\lambda, \lambda \rangle + \zeta^2)^{-\alpha/2}$ and this can be extended to a holomorphic function inside the tube $\{\lambda \in \mathfrak a_{\mathbb C}^\ast: |\Im \lambda| < \zeta\}$. Therefore, by using the theorem above we have the following Sobolev embedding theorem on the fractional order Sobolev spaces $W^{\alpha,p}(X)$.
\begin{cor} \label{cor-sobolev}
Let $1 < p < \infty, \alpha > 0$ and $\zeta > 2|\rho||1/p-1/2|$. Then there exists a positive constant $S_p>0$ such that for all $f\in W^{\alpha, p}(X)$
\bes
\|f\|_p\leq S_p \|(-\Delta- |\rho|^2+\zeta^2)^{\frac{\alpha}{2}} f\|_p.
\ees
\end{cor}
We recall the following analogue of the Poincar\'e-Sobolev inequality
for the fractional Laplace-Beltrami operator on $X$. For proof, we refer the reader to \cite[Theorem 1.11]{BP}.
\begin{thm} \label{thm-p-s}
Let $\dim X=n\geq 3$ and $0<\sigma<\min\{l+2|\Sigma_0^+|, n\}$. Then for $2< p\leq \frac{2n}{n-\sigma}$ there exists $S=S(n, \sigma, p)$ such that for all $u\in W^{\frac{\sigma}{2}, 2}(X)$, 
\bes
\|(-\Delta-|\rho|^2)^{\sigma/4}u\|_2^2 \geq S\|u\|_{p}^2.
\ees
\end{thm}

\section{Optimal Asymptotic Estimates of Bessel-Green-Riesz kernels}
In what follows, $a\lesssim b$ or $a= \mathcal O(b)$ will stand for $a\leq Cb$ with a positive constant $C$ and $a\sim b$ stand for $C^{-1}b\leq a \leq Cb$.

We first set
\be \label{defn-gamma}
\gamma(\alpha) = \frac{2^\alpha~\pi^{n/2}~\Gamma(\alpha/2)}{\Gamma\left((n-\alpha)/2\right)}, \:\:\:\: \txt{ for } 0<\alpha< n. 
\ee
It is well-known that, on the Euclidean spaces, the
following identity of convolution holds \cite[Chapter V]{St}: for 
$\alpha, \beta>0 $ with $\alpha+ \beta < n$
\be \label{gamma-rn}
\int_{\R^n}\|t\|^{\alpha-n}\|t-s\|^{\beta-n}~dt= \frac{\gamma(\alpha) \gamma(\beta)}{\gamma(\alpha+\beta)} ~\|s\|^{\alpha+\beta-n},\:\:\:\: s\in \R^n;
\ee
where the function $\gamma(\eta)$ is defined in (\ref{defn-gamma}) and $\|\cdot\|$ is the Euclidean norm. In order to prove Adams inequality on compact Riemannian manifold (M), L. Fontana proved an analogue of this formula on $M$ \cite[Lemma 2.1]{F}. In this section, we need the following version of this formula valid on a compact subset of $X$. This is essentially proved in \cite{F}, but for the sake of completeness we sketch the proof.
\begin{lem} \label{lem-M}
Suppose $\alpha, \beta>0$ satisfying $\alpha+\beta<n$ and $r>0$. Then there exists $\epsilon$ satisfying $0<\epsilon<\min\{1, n-\alpha-\beta\}$ such that for  $x\in {\bf B}(o, r)$
\bes
\int_{{\bf B}(o, r)} |y|^{\alpha-n} ~|y^{-1}x|^{\beta-n}~dy\leq \frac{\gamma(\alpha)\gamma(\beta)}{\gamma(\alpha+\beta)}~|x|^{\alpha+\beta-n}\left(1+ \mathcal O(|x|^\epsilon)\right).
\ees
\end{lem}
\begin{proof}
For Riemannian symmetric spaces of noncompact type, it is well-known that the sectional curvature is everywhere less than or equals zero \cite[Theorem 3.1, p.241]{H}. On the other hand,  on a compact subset of a Riemannian manifold, the sectional curvature is bounded \cite[Corollary, p.167]{BC}. Therefore, for $r>0$ there exists $\mathcal K_r>0$ such that  for any plane section $P$ at any point $x\in {\bf B}(o, r)$ the sectional curvature $\mathcal K(P)$ satisfies $-\mathcal K_r\leq \mathcal K(P)\leq 0$. 

Let $\mathbb B^n$ be the $n$-dimensional hyperbolic space of constant curvature $-\mathcal K_r$ and $\exp'$ be the corresponding exponential map. Let $o'\in \mathbb B^n$ and $\mathcal B(o', r)$ be the geodesic ball centred at $o'$ and of radius $r$ in $\mathbb B^n$. We consider normal geodesic coordinates on ${\bf B}(o, r)$ and on $\mathcal B(o', r)$ in $X$ and $\mathbb B^n$ respectively. It is a feature of these coordinates that the tangent space at the center is isometric to the standard $n$-dimensional Euclidean space. So, by choosing orthonormal basis in $T_oX=\mathfrak p$ and $T_{o'}\mathbb B^n$, we can identify both the tangent spaces with standard $\R^n$.

If $x$ and $y$ are two points in ${\bf B}(o, r)\subset X$, we consider their normal geodesic coordinates $s$ and $t$ points in $\mathfrak p \cong \R^n$, uniquely determined by $x = \exp(s)$ 
and $y = \exp(t)$. We now construct two points $x', y'$ in $\mathbb B^n$ by $x'=\exp'(s)$ and $y'=\exp'(t)$. The Rauch Comparison Theorem \cite[Theorem 2.3]{F} implies that
\be\label{est-M1}
\|s\|\leq |x|\leq d_{\mathbb B^n}(o', x'), \:\: \|t\|\leq |y|\leq d_{\mathbb B^n}(o', y'), \:\: \textit{ and }\:\:  \|s-t\|\leq |y^{-1} x|\leq d_{\mathbb B^n}(x', y'),
\ee
where $d_{\mathbb B^n}$ is the hyperbolic metric on $\mathbb B^n$. As in \cite[eqn.(15)]{F}, we have that 
\be \label{est-M2}
d_{\mathbb B^n}(x', y')\leq \|s-t\|\left\{
1+\mathcal O\left((\|s\|+ \|t\|)^2\right)\right\},
\ee
where the quantity $\mathcal O\left((\|s\|+ \|t\||^2\right) \leq C_{r, \mathcal K_r} (\|s\|+ \|t\||^2$, for some  $C_{r, \mathcal K_r}>0$ depends only on the radius $r$ and curvature $\mathcal K_r$. Also, we have $dy= \left(1+ \mathcal O(\|t\|)\right)dt$  on the compact set $\overline{{\bf B}(o, r)}$ \cite[Prop. 2.2]{F}. We now choose a small number $\epsilon$ such that $0<\epsilon<\min\{1, n-\alpha-\beta\}$. Then, in normal geodesic coordinates around the origin $o$, using the notations
introduced above, we obtain by (\ref{est-M1}) that 
\beas
\int_{{\bf B}(o, r)} |y|^{\alpha-n} ~|y^{-1}x|^{\beta-n}~dy &\leq& \int_{B(0, r)} \|t\|^{\alpha-n}~\|s-t\|^{\beta-n} \left(1+ \mathcal O(\|t\|)\right)dt\\
&\leq& \int_{B(0, r)} \|t\|^{\alpha-n}~\|s-t\|^{\beta-n} \left(1+ \mathcal O(\|t\|^\epsilon)\right)dt.
\eeas
Therefore, using the Euclidean relation (\ref{gamma-rn}), the estimates (\ref{est-M1}) and (\ref{est-M2}) we have
\beas
&& \int_{{\bf B}(o, r)} |y|^{\alpha-n} ~|y^{-1}x|^{\beta-n}~dy\\
&\leq&\frac{\gamma(\alpha) \gamma(\beta)}{\gamma(\alpha+\beta)} ~\|s\|^{\alpha+\beta-n} ~\left(
1+\mathcal O(\|s\|^\epsilon)\right)\\
&\leq& \frac{\gamma(\alpha) \gamma(\beta)}{\gamma(\alpha+\beta)} ~d_{\mathbb B^n}(o', x')^{\alpha+\beta-n}\left(1+ \mathcal O(d_{\mathbb B^n}(o', x')^\epsilon)\right)\left(1+\mathcal O(\|s\|^2)\right)^{\eta}\\
&\leq& \frac{\gamma(\alpha) \gamma(\beta)}{\gamma(\alpha+\beta)} ~|x|^{\alpha+\beta-n}\left(1+ \mathcal O(|x|^\epsilon)\right),
\eeas
where $\eta$ is a positive number and the quantity $\mathcal O(\|s\|^2)$ is bounded on $B(0, r)$. 
\end{proof}

Let $k_{\zeta, \alpha}$ be the Schwartz  kernel of the operator $(-\Delta -|\rho|^2 +\zeta^2)^{-\alpha/2}$, for $\alpha\in \R$ and $\zeta>0$. Also, let  $k_{\alpha}$ be the kernel of $(-\Delta-|\rho|^2)^{-\alpha/2}$, for $0< \alpha< l+2|\Sigma^+_0|$. The following asymptotic estimates of the Bessel-Green-Riesz kernels at infinity is due to
Anker and Ji \cite[Theorem 4.2.2]{AJ}.
\begin{thm} \label{thm-est-infty}
\begin{enumerate}
\item[(i)] For $\zeta>0$ and $\beta>0$ there holds
\bes 
k_{\zeta, \beta}(x) \sim |x|^{(\beta-l-1)/2-|\Sigma_0^+|}~\phi_0(x)~ e^{-\zeta |x|}, \:\: |x|\geq 1.
\ees
\item[(ii)] For $\zeta=0$ and for $0 < \alpha < l+ 2|\Sigma_0^+|$ there holds
\bes 
k_{\alpha}(x) \sim |x|^{\alpha-l-2|\Sigma_0^+|}~\phi_0(x), \:\: |x|\geq 1.
\ees
\end{enumerate}
\end{thm}
In the remaining part of this section we first derive the optimal bounds for the kernels $k_{\alpha}$, for $0<\alpha<l+2|\Sigma_0^+|$ and $k_{\zeta, \beta}$, for  $\zeta > 0, \beta<n$ near the origin (Proposition \ref{prop-k0} and Proposition \ref{prop-k-zb}). Then we establish sharp
estimates for the convolutions $k_{\alpha} \ast k_{\zeta, \beta}$, for $0 < \alpha+\beta < n$ near the origin and away from the origin using Fourier analysis on symmetric spaces (Proposition \ref{est-0-conv} and Proposition \ref{est-infty-conv}).
\begin{prop} \label{prop-k0}
Let $0< \alpha< \min\{n, l+2|\Sigma_0^+|\}$. There holds 
\bes
k_{\alpha}(x) \leq \frac{1}{\gamma(\alpha)} \frac{1}{|x|^{n-\alpha}}+ \mathcal O\left(\frac{1}{|x|^{n-\alpha-1}}\right), \:\: 0<|x|<1;
\ees
where $\gamma(\alpha)$ is defined in (\ref{defn-gamma})
\end{prop}

\begin{proof}
By the Mellin type expression
\be \label{mellin}
(-\Delta-|\rho|^2)^{-\alpha/2}= \frac{1}{\Gamma(\alpha/2)}\int_{0}^\infty t^{\alpha/2-1}e^{-t(-\Delta -|\rho|^2)}~dt. 
\ee
We will now use the following local expansion of the heat kernel $h_t(x)$ 
\bes
h_t(x)=e^{-|x|^2/4t} t^{-n/2} v_0(x) + \mathcal O \left(e^{-c|x|^2/t}t^{-n/2+1}\right),
\ees
where $v_0(x)=(4\pi)^{-n/2}+ \mathcal O (|x|^2)$ and $0<c<1/4$ \cite[eqn.(3.9), p.278]{An2}. Using this it follows from (\ref{mellin}) that on the kernel level
\bea \label{eqn-k0}
k_{\alpha}(x)&=& \frac{1}{\Gamma(\alpha/2)} \int_{0}^\infty t^{\alpha/2-1}~h_t(x)~e^{|\rho|^2 t}~dt \nonumber\\
&=& \frac{1}{\Gamma(\alpha/2)} \int_{0}^1 t^{\alpha/2-1}~\left(e^{-|x|^2/4t} t^{-n/2} v_0(x) + \mathcal O\left(e^{-c|x|^2/t}t^{-n/2+1}\right)\right) e^{|\rho|^2 t}~dt \nonumber\\
&& +~\frac{1}{\Gamma(\alpha/2)} \int_{1}^\infty t^{\alpha/2-1}~h_t(x)~e^{|\rho|^2 t}~dt \nonumber \\
&=& \frac{1}{\Gamma(\alpha/2)} \int_{0}^1 t^{\alpha/2-1}~\left(e^{-|x|^2/4t}~ t^{-n/2}~v_0(x) + \mathcal O\left(e^{-c|x|^2/t}t^{-n/2+1}\right)\right) \left(1+ \mathcal O(t)\right)~dt \nonumber\\
&& +~\frac{1}{\Gamma(\alpha/2)} \int_{1}^\infty t^{\alpha/2-1}~h_t(x)~e^{|\rho|^2 t}~dt \nonumber\\
&\leq& \frac{1}{\Gamma(\alpha/2)}~v_0(x) \int_{0}^1 e^{-|x|^2/4t}~t^{\alpha/2-1-n/2}~dt+ \int_{0}^1  \mathcal O\left(e^{-c|x|^2/t}~t^{\alpha/2-n/2}\right)~dt \nonumber\\
&& +~\frac{1}{\Gamma(\alpha/2)} \int_{1}^\infty t^{\alpha/2-1}~h_t(x)~e^{|\rho|^2 t}~dt.
\eea
Now, we have that
\bea \label{est-1}
\int_{0}^1 e^{-|x|^2/4t}~t^{\alpha/2-1-n/2}~dt &=& \left(\frac{2}{|x|}\right)^{n-\alpha}\int_{|x|^2/4}^\infty e^{-s}~s^{(n-\alpha)/2-1}~ds \nonumber\\
&\leq& \left(\frac{2}{|x|}\right)^{n-\alpha}~\Gamma\left(\frac{n-\alpha}{2}\right).
\eea
Similarly, the second integral is of $\mathcal O(|x|^{\alpha-n+1})$. For the third term, we use the estimate on $h_t$ \cite[Theorem 3.1, ii)]{An2} that there exists $C>0$ such that
\bes 
h_t(x)\leq Ct^{-l/2-|\Sigma_0^+|}~(1+|x|^2)^{|\Sigma_0^+|/2}~e^{-|\rho|^2 t-\rho(\log x)-|x|^2/(4t)}, \:\: t\geq 1, \:\: |x|\leq \sqrt t.
\ees
Using this estimate it follows that for all $0<|x|<1$ and $0<\alpha<l+2|\Sigma_0^+|$
\be \label{est-2}
\frac{1}{\Gamma(\alpha/2)} \int_{1}^\infty t^{\alpha/2-1}~h_t(x)~e^{|\rho|^2 t}~dt \leq C \int_{1}^\infty t^{\alpha/2-1-l/2-|\Sigma_0^+|}~dt<\infty.
\ee
Using $v_0(x)=(4\pi)^{-n/2}+ \mathcal O (|x|^2)$ and the estimates (\ref{est-1}) and (\ref{est-2}),
it follows from the equation (\ref{eqn-k0}) that for $0<|x|<1$
\beas
k_{\alpha}(x) &\leq& \frac{1}{\Gamma(\alpha/2)}~(4\pi)^{-n/2}~\Gamma\left((n-\alpha)/2\right)~2^{n-\alpha}~|x|^{\alpha-n}+ \mathcal O(|x|^{\alpha-n+1})+ \mathcal O(1)\\
&\leq& \frac{1}{\gamma(\alpha)}~|x|^{\alpha-n} + \mathcal O(|x|^{\alpha-n+1}).
\eeas
This completes the proof.
\end{proof}

\begin{prop} \label{prop-k-zb}
Let $\zeta>0$ and $0< \beta< n$. Then there exists $\tilde\epsilon$ satisfying $0<\tilde\epsilon< \min\{1, n-\beta\}$ such that  
\bes
k_{\zeta,\beta}(x)= \frac{1}{\gamma(\beta)} \frac{1}{|x|^{n-\beta}} + \mathcal O\left(\frac{1}{|x|^{n-\beta-\tilde\epsilon}}\right), \:\: 0< |x|< 1.
\ees
\end{prop}
\begin{proof}
We first prove the result when $\beta=m$ is an integer satisfying $1\leq m < n-1$. Precisely, we prove that there exists $\epsilon_0$ satisfying $0< \epsilon_0< \min\{1, n-m\}=1$ such that
\be \label{eqn-1000}
k_{\zeta, m}(x)\leq \frac{1}{\gamma(m)} \frac{1}{|x|^{n-m}} + \mathcal O\left(\frac{1}{|x|^{n-m-\epsilon_0}}\right), \:\: 0< |x|< 1.
\ee
This will be done by induction. It follows by the Mellin type expressions that on the kernel level	
\beas
k_{\zeta, 1}(x)&=& \frac{1}{\Gamma(1/2)} \int_{0}^\infty t^{-1/2}~h_t(x)~e^{\rho^2 t-\zeta^2 t}~dt \nonumber\\
&\leq & \frac{1}{\Gamma(1/2)} \int_{0}^\infty t^{-1/2}~h_t(x)~e^{\rho^2 t}~dt \nonumber\\
&=& k_{1}(x).
\eeas
Thus, by Proposition \ref{prop-k0}, the estimate (\ref{eqn-1000}) holds for $m=1$. In fact, this is true as long as $m <\min\{n, l+2|\Sigma_0^+|\}$ by Proposition \ref{prop-k0}. Now, suppose that the estimate (\ref{eqn-1000}) is valid for $m\in \Z$ satisfying $1\leq m < n-2$ with $0< \epsilon_0<\min\{1, n-m\}=1$. We prove this is also true for $m+1$, that is, there exists $\tilde\epsilon$ satisfying $0< \tilde\epsilon< \min\{1, n-m-1\}$ such that 
\be \label{eqn-100}
k_{\zeta, m+1}(x)\leq \frac{1}{\gamma(m+1)} \frac{1}{|x|^{n-m-1}} + \mathcal O\left(\frac{1}{|x|^{n-m-1-\tilde\epsilon}}\right), \:\: 0< |x|< 1.
\ee
 Let $x\in X$ with $0< |x|< 1$. Then
\be \label{k0*k-est}
k_{1}\ast k_{\zeta, m}(x)= \int_{{\bf B}(o, 2)} k_{1}(y)~k_{\zeta,m}(y^{-1}x)~dy+ \int_{X\backslash {\bf B}(o, 2)} k_{1}(y)~k_{\zeta,m}(y^{-1}x)~dy :=I_1+I_2.
\ee
We first prove that the second integral $I_2$ on the right-hand side is uniformly bounded independent of $x$. Indeed,  by H\"older's inequality, Theorem \ref{thm-est-infty}, integral formula (\ref{int-formula}) with (\ref{J-est}) and the estimate (\ref{est-phi0}) of $\phi_0$, it follows that for $|x|<1$
\beas \label{bound-2nd}
I_2 &=&\int_{X\backslash {\bf B}(o, 2)} k_{1}(y)~k_{\zeta,m}(y^{-1}x)~dy\\
&\leq& \left(\int_{X\backslash {\bf B}(o, 2)} |k_{1}(y)|^2 ~dy\right)^{1/2}~\left(\int_{X\backslash {\bf B}(o, 1)} |k_{\zeta, m}(z)|^2~dz\right)^{1/2}\\
&\lesssim& \left( \int_{X\backslash {\bf B}(o, 2)}|y|^{2-2l-4|\Sigma_0^+|}~\phi_0(y)^2~dy \right)^{1/2}
\left(\int_{X\backslash {\bf B}(o, 1)}|z|^{m-l-1-2|\Sigma_0^+|}~\phi_0(z)^2~ e^{-2\zeta |z|}~dz \right)^{1/2}\\
&\lesssim&  \left( \int_{\{Y\in \overline{\mathfrak a_+}: |Y|\geq 2\}} |Y|^{2-2l-4|\Sigma_0^+|}~|Y|^{2|\Sigma_0^+|} e^{-2\rho(Y)}~e^{2\rho(Y)}~dY \right)^{1/2}\\
&& \left(\int_{\{Y\in \overline{\mathfrak a_+}: |Y|\geq 1\}} |Y|^{m-l-1-2|\Sigma_0^+|}~|Y|^{2|\Sigma_0^+|} e^{-2\rho(Y)}~ e^{-2\zeta |Y|}~e^{2\rho(Y)}~dY \right)^{1/2}\\
&\lesssim& \left(\int_{2}^\infty r^{2-2l-2|\Sigma_0^+|}~r^{l-1}~dr\right)^{1/2}~\left(\int_{1}^\infty r^{m-l-1}~e^{-2\zeta r}~r^{l-1}~dr \right)^{1/2}< \infty.
\eeas
Now, we estimate the first integral $I_1$ in (\ref{k0*k-est}). Since $m< n-2$ and $\epsilon_0<1$, we can choose a $\epsilon_1 \in (0, 1)$ such that $m+\epsilon_0+\epsilon_1< n-1$. We first observe that $|y|^{-(n-2)}\leq |y|^{-(n-1-\epsilon_1)}$, for small $|y|$. Using this fact, the estimate of $k_{\alpha}$ in Proposition \ref{prop-k0} and the estimate  (\ref{eqn-1000}) of $k_{\zeta, m}$ we get that 
\beas \label{eqn-k0-kz-0}
I_1 &=& \int_{{\bf B}(o, 2)} k_{1}(y)~ k_{\zeta, m}(y^{-1}x)~dy \nonumber\\
&\leq& \int_{{\bf B}(o, 2)} \left(\frac{1}{\gamma(1)} \frac{1}{|y|^{n-1}} + \frac{C_1}{|y|^{n-1-\epsilon_1}} \right)~\left(\frac{1}{\gamma(m)} \frac{1}{|y^{-1}x|^{n-m}}+ \frac{C_2}{|y^{-1}x|^{n-m-\epsilon_0}}\right)~dy\\
&=& \frac{1}{\gamma(1)\gamma(m)} \int_{{\bf B}(o, 2)} \frac{1}{|y|^{n-1}}~\frac{1}{|y^{-1}x|^{n-m}}~dy + \frac{C_2}{\gamma(1)} \int_{{\bf B}(o,2)} \frac{1}{|y|^{n-1}} \frac{1}{|y^{-1} x|^{n-m-\epsilon_0}}~dy \\
&& + \frac{C_1}{\gamma(m)}\int_{{\bf B}(o, 2)} \frac{1}{|y|^{n-1-\epsilon_1}}~\frac{1}{|y^{-1} x|^{n-m}} ~dy + C_1 C_2 \int_{{\bf B}(o, 2)} \frac{1}{|y|^{n-1-\epsilon_1}}~\frac{1}{|y^{-1} x|^{n-m-\epsilon_0}} ~dy.
\eeas
Since $m +1+\epsilon_1+\epsilon_0< n$, using Lemma \ref{lem-M} we get that there exists $0<\tilde \epsilon< \min\{1, n-(m+1+\epsilon_0+\epsilon_1)\}$ such that for $|x|<1$
\bes
I_1\leq\frac{1}{\gamma(m+1)}~ \frac{1}{| x|^{n-m-1}}+ \mathcal O\left(\frac{1}{|x|^{n-m-1-\tilde\epsilon}}\right).
\ees
Putting this in (\ref{k0*k-est}) we finally have for $|x|<1$
\bes
k_{\zeta, m+1}(x)= k_{\zeta, 1} \ast k_{\zeta, m}(x)\leq k_{1} \ast k_{\zeta, m}(x)\leq \frac{1}{\gamma(m+1)}~ \frac{1}{|x|^{n-m-1}}+ \mathcal O\left(\frac{1}{|x|^{n-m-1-\tilde \epsilon}}\right).
\ees
By induction this completes the proof of the lemma when $\beta=m\in \Z$ with $1\leq m  < n-1$. 

Now, we prove the required estimate for arbitrary $\beta$ (not necessarily integer) with $0< \beta< n$.
We choose $0< \tilde \beta< 3$ and an integer $m$ with $0\leq m< n-1$ such that $\beta= \tilde \beta +m< n$. Without loss of generality, we can assume $m\geq 1$. Thus
\bea \label{est-epsilon}
k_{\zeta, \beta}(x) &=& k_{\zeta, \tilde \beta} \ast k_{\zeta, m}\leq k_{\tilde \beta} \ast k_{\zeta, m}\nonumber\\
&=& \int_{{\bf B}(o, 2)} k_{\tilde \beta}(y)~k_{\zeta,m}(y^{-1}x)~dy+ \int_{X\backslash {\bf B}(o, 2)} k_{\tilde \beta}(y)~k_{\zeta,m}(y^{-1}x)~dy.
\eea
The second integral is bounded as in the case of $I_2$ in (\ref{k0*k-est}). The first integral can be estimated as $I_1$ in (\ref{k0*k-est}). To see this, we first notice that we may choose $\epsilon_0$ small enough satisfying $0<\epsilon_0< \min\{1, n-\beta\}$ such that (\ref{eqn-1000}) holds. Let us choose $\epsilon_2$ such that $0<\epsilon_2<\min\{1, n-\beta-\epsilon_0\}$. By Proposition \ref{prop-k0} it follows that
\beas
&& \int_{{\bf B}(o, 2)} k_{\tilde \beta}(y)~k_{\zeta,m}(y^{-1}x)~dy \\
&\leq& \int_{{\bf B}(o, 2)} \left(\frac{1}{\gamma(\tilde \beta)} \frac{1}{|y|^{n-\tilde \beta}} + \frac{C_1}{|y|^{n-\tilde \beta-\epsilon_2}} \right)~\left(\frac{1}{\gamma(m)} \frac{1}{|y^{-1}x|^{n-m}}+ \frac{C_2}{|y^{-1}x|^{n-m-\epsilon_0}}\right)~dy.
\eeas
Therefore, by Lemma \ref{lem-M} there exists $\tilde\epsilon$ satisfying $0< \tilde \epsilon<\min\{1, n-\beta-\epsilon_0-\epsilon_2\}$ such that
\bes
\int_{{\bf B}(o, 2)} k_{\tilde \beta}(y)~k_{\zeta,m}(y^{-1}x)~dy \leq \frac{1}{\gamma(\beta)} \frac{1}{|x|^{n-\beta}}+\mathcal O\left(\frac{1}{|x|^{n-\beta-\tilde\epsilon}}\right).
\ees 
Putting this in (\ref{est-epsilon}) we complete the proof.  
\end{proof}	

\begin{prop} \label{est-0-conv}
Let $\zeta>0, 0< \alpha<l+2|\Sigma_0^+|$ and $0< \beta< n$ such that $0< \alpha+\beta<n$. There exists $\epsilon'$ satisfying $0< \epsilon'< \min\{1, n-\alpha-\beta\}$  such that
\bes
k_{\alpha} \ast k_{\zeta,\beta}(x)\leq \frac{1}{\gamma(\alpha+\beta)} \frac{1}{|x|^{n-\alpha-\beta}} + \mathcal O\left(\frac{1}{|x|^{n-\alpha-\beta-\epsilon'}}\right), \:\: 0< |x|< 1.
\ees
\end{prop}
\begin{proof}
The proof is exactly the same as that of Proposition \ref{prop-k-zb}.
\end{proof}

\begin{prop} \label{est-infty-conv}
Let $\zeta>0, 0< \alpha<l+2|\Sigma_0^+|$ and $0< \beta< n$ such that $0< \alpha+\beta<n$. For $\zeta'\in (0, \zeta)$ we have
\bes
k_{\alpha} \ast k_{\zeta, \beta}(x) \lesssim e^{-\zeta'|x|}~\phi_0(x) + \left(\chi_{1/2} |\cdot|^{\alpha-l-2|\Sigma_0^+|}~\phi_0(\cdot)\right)\ast k_{\zeta, \beta} (x), \:\: \textit{ for }| x|\geq 1,
\ees
where $\chi_{1/2}$ is the cutoff function vanishing in ${\bf B}(o, 1/2)$ and identically equals $1$ otherwise.
\end{prop}
\begin{proof}
By Theorem \ref{thm-est-infty} (ii) we have
\bea \label{est-k*k-infty}
k_{\alpha} \ast k_{\zeta, \beta}(x) 
&=& \int_{{\bf B}(o, 1/2)} k_{\alpha}(y) k_{\zeta, \beta}(y^{-1}x)~dy + \int_{X\backslash {\bf B}(o, 1/2)} k_{\alpha}(y) k_{\zeta, \beta}(y^{-1}x)~dy \nonumber\\
&\lesssim&  \int_{{\bf B}(o, 1/2)} k_{\alpha}(y) k_{\zeta, \beta}(y^{-1}x)~dy + \int_{X\backslash {\bf B}(o, 1/2)} |y|^{\alpha-l-2|\Sigma_0^+|} \phi_0(y) ~ k_{\zeta, \beta}(y^{-1}x)~dy\nonumber\\
&=& \int_{{\bf B}(o, 1/2)} k_{\alpha}(y) k_{\zeta, \beta}(y^{-1}x)~dy + \left(\chi_{1/2} |\cdot|^{\alpha-l-2|\Sigma_0^+|}~\phi_0(\cdot)\right) \ast k_{\zeta, \beta} (x).
\eea
We notice that, if $|y|< 1/2$ and $|x|\geq 1$, then $|y^{-1}x|\geq |x|-|y| \geq 1/2$. Therefore, by the estimates of $k_{\alpha}$ (Proposition \ref{prop-k0}) and $k_{\zeta, \beta}$ (Theorem \ref{thm-est-infty} (i)) we have for $|y|< 1/2$ and $|x|\geq 1$ that
\bes
k_{\alpha}(y)\lesssim  |y|^{\alpha-n},\:\: \textit{ and } \:\: 
k_{\zeta, \beta}(y^{-1}x)\lesssim  e^{-\zeta '|y^{-1}x|} \phi_0(y^{-1}x),
\ees
where $\zeta^\prime\in (0, \zeta)$ and the constant depends on $\zeta^\prime$. Using the above estimates, the property (\ref{phi-xy}), the fact $|H(y^{-1}k)|\leq |y|$ (see equation (\ref{Iwasawa})) and the integral formula (\ref{int-formula}), it follows that
\beas
&&\int_{{\bf B}(o, 1/2)} k_{\alpha}(y) k_{\zeta, \beta}(y^{-1}x)~dy \\
&\lesssim &  \int_{{\bf B}(o, 1/2)} \frac{1}{|y|^{n-\alpha}} e^{-\zeta' |y^{-1}x|}~\phi_0(y^{-1}x)~dy\\
&\leq&  e^{-\zeta'|x|}\int_{{\bf B}(o, 1/2)} \frac{1}{|y|^{n-\alpha}} ~ e^{\zeta'|y|}
~\int_K e^{-\rho\big(H(y^{-1}k) \big)}e^{-\rho(H(x^{-1}k))}~dk~dy\\
&\leq&  e^{-\zeta'|x|} \int_K e^{-\rho H(x^{-1}k)}~dk~ \int_{{\bf B}(o, 1/2)} \frac{1}{|y|^{n-\alpha}}~e^{\zeta' |y|}~e^{|\rho| |y|}~dy\\
&\lesssim &  e^{-\zeta'|x|}~\phi_0(x) \int_{\{Y\in \mathfrak a: |Y|<1/2\}} \frac{1}{|Y|^{n-\alpha}}~J(\exp Y)~dY\\
&\lesssim &  e^{-\zeta'|x|}~\phi_0(x).
\eeas
Putting this in (\ref{est-k*k-infty}) we get the required result. 
\end{proof}
\begin{rem} \label{rem-k-conv}
There exists $C>0$ such that $k_{\alpha}\ast k_{\zeta, \beta}(x)\leq C$, for all $|x| \geq 1$. Indeed, by Proposition \ref{est-infty-conv} it is enough to show that $\left(\chi_{1/2} |\cdot|^{\alpha-l-2|\Sigma_0^+|}~\phi_0(\cdot)\right)\ast k_{\zeta, \beta} (x)\leq C$, for $|x|\geq 1$. To see this we first observe that if $|x|\geq 1$ and $|y|< 1/2$, then $|y^{-1}x|\geq 1/2$. Using this, the property (\ref{phi-xy}), the estimate of $k_{\zeta, \beta}$ (Theorem \ref{thm-est-infty} and Proposition \ref{prop-k-zb}) it follows by Cauchy-Schwarz inequality that
\beas
&&\left(\chi_{1/2} |\cdot|^{\alpha-l-2|\Sigma_0^+|}~\phi_0(\cdot)\right)\ast k_{\zeta, \beta} (x)\\
&\lesssim& \int_{X\backslash {\bf B}(o, 1/2)} |y|^{\alpha-l-2|\Sigma_0^+|} \phi_0(y) k_{\zeta, \beta} (y^{-1}x)~dy\\
&\lesssim& \int_{\{y\in X \backslash {\bf B}(o, 1/2): ~|y^{-1}x|< 1/2\}} |y|^{\alpha-l-2|\Sigma_0^+|} \phi_0(y)|y^{-1}x|^{\beta-n}~dy\\
&& + \int_{\{y\in X \backslash {\bf B}(o, 1/2): ~|y^{-1}x|\geq 1/2\}} |y|^{\alpha-l-2|\Sigma_0^+|} \phi_0(y) k_{\zeta, \beta} (y^{-1}x)~dy\\
&\lesssim&  \int_{\{y\in X: ~|y^{-1}x|< 1/2\}} |y^{-1}x|^{\beta-n}~dy + \int_{\{y\in X \backslash {\bf B}(o, 1/2): ~|y^{-1}x|\geq 1/2\}} |y|^{\alpha-l-2|\Sigma_0^+|} \phi_0(y) ~k_{\zeta, \beta}(y^{-1}x)~dy\\
&\lesssim& C+\left( \int_{X\backslash {\bf B}(o, 1/2)} |y|^{2\alpha-2l-4|\Sigma_0^+|}~(\phi_0(y))^2~dy\right)^{1/2} \left(\int_{X\backslash {\bf B}(o, 1/2)}e^{-2\zeta^\prime |z|}~\left(\phi_0(z)\right)^2~dz\right)^{1/2}< \infty.
\eeas 
\end{rem}

\section{Asymptotic estimates of non-increasing rearrangement of the Bessel-Green-Riesz kernels}

For a real valued function $f$ on $X$, the non-increasing rearrangement of $f$ is defined by
\bes
f^\ast(t)= \inf \{s>0: \lambda_f(s)\leq t\},
\ees
where the distribution function $\lambda_f$ of $f$ is given by
\bes
\lambda_f(s)=|\{x\in X: |f(x)|>s\}|= \int_{\{x\in X: |f(x)|>s\}} dx.
\ees
Here we use the notation $|E|$ for the measure of a subset $E$ of $X$. We need the following two properties of the non-increasing rearrangement \cite[Prop. 1.4.5, p.46]{Gf}:
\begin{enumerate}
\item[(i)] By definition $\la_f(f^\ast(t))\leq t$. If $|f|\leq |g|$ almost everywhere, then $f^\ast\leq g^\ast$.

\item[(ii)] If there exists $c>0$ such that $|\{x\in X: |f(x)|\geq f^\ast(t)-c\}|< \infty$, then $t\leq |\{x\in X: |f(x)|\geq f^\ast(t)\}|$.
\end{enumerate} 
For the convenience of the reader we summarize the results of section 3 here. By Theorem \ref{thm-est-infty}, Proposition \ref{prop-k0}, Proposition \ref{prop-k-zb}, Proposition \ref{est-0-conv} and Proposition \ref{est-infty-conv}  we have
\begin{enumerate}
\item[a)] For $\zeta=0$ and $0< \alpha< \min\{n, l+2|\Sigma_0^+|\}$
\beas
k_{\alpha}(x) &\leq& \frac{1}{\gamma(\alpha)}\frac{1}{|x|^{n-\alpha}} +\mathcal O\left(\frac{1}{|x|^{n-\alpha-1}}\right), \:\:  0<|x|<1;\\
&\lesssim& |x|^{\alpha-l-2|\Sigma_0^+|}~\phi_0(x), \:\: |x|\geq 1. 
\eeas

\item[b)] Let $\zeta>0$ and $0< \beta<n$. There exists $\tilde\epsilon$ satisfying $0<\tilde\epsilon< \min\{1, n-\beta\}$ such that  
\bea \label{beta-est}
k_{\zeta, \beta}(x) &\leq& \frac{1}{\gamma(\beta)} \frac{1}{|x|^{n-\beta}} + \mathcal O\left(\frac{1}{|x|^{n-\beta-\tilde\epsilon}}\right), \:\:  0< |x|< 1;\nonumber\\
&\lesssim & |x|^{(\beta-l-1)/2-|\Sigma_0^+|}~e^{-\zeta |x|}~\phi_0(x), \:\: |x|\geq 1.
\eea

\item[c)] Let $\zeta>0, 0< \alpha<l+2|\Sigma_0^+|$ and $0< \beta< n$ such that $0< \alpha+\beta<n$. There exists $\epsilon'$ satisfying $0< \epsilon^\prime < \min\{1, n-\alpha-\beta\}$  such that
\bea \label{conv-est-infty}
k_{\alpha} \ast k_{\zeta, \beta}(x) &\leq& \frac{1}{\gamma(\alpha+\beta)}\frac{1}{|x|^{n-\alpha-\beta}} + \mathcal O\left(\frac{1}{|x|^{n-\alpha-\beta-\epsilon^\prime}}\right), \:\: 0<|x|< 1. \nonumber\\
&\lesssim&  e^{-\zeta'|x|}~\phi_0(x) + \left(\chi_{1/2} |\cdot|^{\alpha-l-2|\Sigma_0^+|}~\phi_0(\cdot)\right)\ast k_{\zeta, \beta} (x), \:\: \textit{ for }| x|\geq 1,
\eea
where $0< \zeta'< \zeta$.
\end{enumerate}

The behaviour of the volume of a small geodesic ball around the origin in $X$ can be expressed as follows (see \cite[Theorem 3.98]{GHL}, \cite[equation (8)]{K}):
\bea \label {vol}
|{\bf B}(o, r)| &=& \frac{\omega_{n-1}}{n}~r^n +\mathcal O(r^{n+1}), \:\: 0< r<1.
\eea

\begin{lem} \label{lem-k-ast}
Let $\zeta, \alpha, \beta, \epsilon^\prime$ be as in Proposition \ref{est-0-conv}. Then for $0< t< 2$
\bes
[k_{\alpha} \ast k_{\zeta, \beta}]^\ast(t)\leq \frac{1}{\gamma(\alpha+\beta)} \left(\frac{nt}{\omega_{n-1}}\right)^{(\alpha+\beta-n)/n} + \mathcal O\left(t^{(\alpha+\beta+\epsilon^\prime-n)/n}\right).
\ees
\end{lem}
\begin{proof}
By (\ref{conv-est-infty}) there exists $C>0$ such that for $0<|x|< 1$, $k_{\alpha}\ast k_{\zeta, \beta}(x)\leq h(x)$, where the function $h$
is defined by
\bes
h(x)= \frac{1}{\gamma(\alpha+\beta)} \frac{1}{|x|^{n-\alpha-\beta}} + \frac{C}{|x|^{n-\alpha-\beta-\epsilon^\prime}}, \:\: x\in X.
\ees 
By Remark \ref{rem-k-conv} there exists $C_0>0$ such that $k_{\alpha}\ast k_{\zeta, \beta}(x)\leq C_0$ for $|x|\geq 1$. Therefore, $k_{\alpha}\ast k_{\zeta, \beta}(x)\leq C_0+h(x)$ for all $x\in X$ and hence it is enough to prove the required estimate for the function $h^\ast$.

First, we note that if $f(t) = At^{-a}(1 + Bt^b), t > 0$, for positive constants $A, B, a, b$ then there exists a $C > 0$ such that
\bes
f^{-1}(t)\leq [At^{-1}]^{1/a}[1+Ct^{-b/a}], \txt{ for } t>1.
\ees
Applying this to the function $f(|x|)=h(x)$,  we get that
\bes
h^{-1}(t)\leq \left[\frac{t^{-1}}{\gamma(\alpha+\beta)}\right]^{1/(n-\alpha-\beta)}~\left[1+Ct^{-\epsilon^\prime/(n-\alpha-\beta)}\right], \:\: t>1.
\ees 
Using the above inequality and (\ref{vol}) we get that for $t>1$ 
\beas
|\{x\in X: h(x)\geq t\}|\leq |{\bf B}\left(o, h^{-1}(t)\right)| &\leq& \frac{\omega_{n-1}}{n} \left(h^{-1}(t)\right)^n + C(h^{-1}(t))^{n+1}\\
&\leq& \frac{\omega_{n-1}}{n}~\left(\frac{t^{-1}}{\gamma(\alpha+\beta)}\right)^{n/(n-\alpha-\beta)}~ \left[1+ C^\prime t^{-\epsilon^\prime/(n-\alpha-\beta)}\right].
\eeas
Again, if $g(t) = At^{-a}[1 + Bt^{-b}]$, $t > 0$, for positive constants $A, B, a, b$, then there exists a $C > 0$ such that 
\bes
g^{-1}(t)\leq [At^{-1}]^{1/a}[1+Ct^{b/a}], \txt{ for } 0<t\leq 2.
\ees
Using this we get that for $0< t\leq 2$
\bes
h^\ast(t)\leq \frac{1}{\gamma(\alpha+\beta)} \left(\frac{\omega_{n-1}}{nt}\right)^{(n-\alpha-\beta)/n} ~[1+\mathcal O(t^{\epsilon^\prime/n})].
\ees
This completes the proof.
\end{proof}


\begin{lem} \label{lem-beta-ast}
Let $\zeta, \beta, \tilde \epsilon$ be as in Proposition \ref{prop-k-zb}. Then there holds
\bes
[k_{\zeta, \beta}]^\ast(t)\leq \frac{1}{\gamma(\beta)} \left(\frac{nt}{\omega_{n-1}}\right)^{(\beta-n)/n} +\mathcal O(t^{(\beta+\tilde\epsilon-n)/n}), \:\:  0<t<2.
\ees
\end{lem}
\begin{proof}
The proof is similar to that of Lemma \ref{lem-k-ast}.
\end{proof}


\begin{lem} \label{lem-est-k-ast-inft}
Let $\zeta>0$, $0<\beta<n$ and $\zeta^\prime\in(0, \zeta)$. There holds
\bes
[k_{\zeta, \beta}]^\ast(t)\lesssim  t^{-1/2-\zeta'/2|\rho|} ~(\log t)^{2|\rho|l/(\zeta'+|\rho|)}, \:\: t\geq 2.
\ees
\end{lem}
\begin{proof}
Let us choose $r>0$ such that $|{\bf B}(o, r)|< 1$ and set
\beas
h(x)&=&\frac{1}{|x|^{n-\beta}}, \:\: 0<|x|<r;\\
&=& 0, \:\:\: \hspace{1cm} |x|\geq 1,
\eeas
and $f(x)= e^{-\rho(\log  x)}~ e^{-\zeta^\prime |x|}$, for $x\in X$.
Using the estimate (\ref{est-phi0}) of $\phi_0$, it follows form (\ref{beta-est}) that there exists $C>0$ such that
\be \label{eqn-k-beta-ast}
k_{\zeta, \beta}(x) \leq C \left(h(x)+f(x)\right), \:\:\:\: x\in X.
\ee
We now observe that for $s>0$
\bes
|\{x\in X: h(x)> s\}|\leq |{\bf B}(o, r).
\ees
This inequality yields 
\beas
|\{x\in X: h(x)+ f(x)> s\}| &\leq& |\{x\in X: h(x)> s/2\}| +|\{x\in X: f(x)> s/2\}|\\
&\leq& |{\bf B}(o, r)|+ |\{x\in X:f(x)> s/2\}|.
\eeas
Therefore, for $t> |{\bf B}(o, r)|$ we have
\bea \label{eqn-conv1}
\left(h+f\right)^\ast(t)
&=& \inf_{s}\{s>0: |\{x\in X: h(x)+ f(x)>s\}|\leq t\} \nonumber\\
&\leq & \inf_{s}\{s>0: |\{x\in X:f(x)>s/2\}|\leq t-|{\bf B}(o, r)|\}\nonumber\\
&=& 2f^\ast\left(t-|{\bf B}(o, r)|\right).
\eea
Integral formula (\ref{int-formula}) yields
\bea \label{est-t}
\lambda_{f}\left(f^\ast(t)\right)&=& \int_{\{Y\in\overline{\mathfrak a_+}: f(Y)> f^\ast(t)\}} c \prod_{\alpha\in \Sigma_+}\left(\sinh\alpha(Y)\right)^{m_{\alpha}}~dY\nonumber\\
&\lesssim& \int_{\{Y\in\overline{\mathfrak a_+}:~ e^{\rho(Y)+\zeta^\prime |Y|}< \frac{1}{f^\ast(t)}\}} ~e^{2\rho(Y)}~dY.
\eea
Let us fix a basis of $\mathfrak a^\ast$ as  $\{\epsilon_1, \cdots, \epsilon_{l-1}, \rho/ |\rho|\}$, where $\rho^{\perp}= span\{\epsilon_1, \cdots, \epsilon_{l-1}\}$. If $Y\in \mathfrak a$, we write $Y=(Y_1, \cdots, Y_l)$ with its coordinates with respect to the corresponding dual basis. We observe that $\rho(Y)=|\rho| Y_l$. Since $\rho(Y)\geq 0$ for $Y\in \overline{\mathfrak a_+}$, it follows that
\beas
\left\{Y\in\overline{\mathfrak a_+}:~ e^{\rho(Y)+\zeta^\prime|Y|}< \frac{1}{f^\ast(t)} \right\} &\subset& \left\{Y\in\overline{\mathfrak a_+}: e^{\rho(Y)+\zeta^\prime Y_l}< \frac{1}{f^\ast(t)}, ~ e^{\zeta^\prime |Y|}< \frac{1}{f^\ast(t)} \right\}\\
&\subset& \left\{Y\in\overline{\mathfrak a_+}: e^{|\rho|Y_l + \zeta^\prime Y_l}< \frac{1}{f^\ast(t)}, ~ |Y|< \frac{1}{\zeta^\prime}\log\left(\frac{1}{f^\ast(t)}\right)\right\}.
\eeas
Therefore, by (\ref{est-t}) we get that
\beas
\lambda_{f}\left(f^\ast(t)\right) &\lesssim& \int_{\{Y\in\overline{\mathfrak a_+}: e^{|\rho|Y_l+\zeta^\prime Y_l}< \frac{1}{f^\ast(t)}, ~ |Y|< \frac{1}{\zeta_1}\log\left(\frac{1}{f^\ast(t)}\right)\}}~e^{2|\rho|Y_l}~dY\\
&\lesssim& \left(\frac{1}{f^\ast(t)}\right)^{2|\rho|/\left(\zeta^\prime +|\rho|\right)} \left(\log \left(\frac{1}{f^\ast(t)}\right)\right)^l.
\eeas
Since $f^\ast(t)$ is non-zero it follows from the above estimate that $\la_f(f^\ast(t-c))< \infty$ for $c\in (0, f^\ast(t))$. Therefore, 
\bes
t= \la_f(f^\ast(t))\lesssim \left(\frac{1}{f^\ast(t)}\right)^{2|\rho|/\left(\zeta^\prime +|\rho|\right)} \left(\log\left(\frac{1}{f^\ast(t)}\right)\right)^l.
\ees
Using the lemma below (Lemma \ref{lem-fol}) we get that for all $t\geq 1$
\bes
f^\ast(t)\lesssim t^{-(\zeta'+|\rho|)/2|\rho|}~(\log t)^{l(\zeta'+|\rho|)/2|\rho|}. 
\ees
This fact together with (\ref{eqn-k-beta-ast}) and (\ref{eqn-conv1}) completes the proof. 
\end{proof}
\begin{lem} \label{lem-fol}
Let $h$ be a non-increasing function on the positive real axis. Suppose there exist two positive numbers $a, b$ such that for $t\geq 1$
\bes
t\leq \left(\frac{1}{h(t)}\right)^a~\log\left(\frac{1}{h(t)}\right)^b.
\ees
Then there holds
\bes
h(t)\leq t^{-1/a}~(\log t)^{b/a}, \:\: \textit{ for } t\geq 1. 
\ees
\end{lem}
\begin{proof}
Let $h(t)=s$. By the hypothesis we have for all $t\geq 1$
\bes
h^{-1}(s)=t\lesssim \frac{1}{s^a} \left(\log \frac{1}{s}\right)^b. 
\ees 
Therefore, for $s\geq 1$
\beas
h^{-1}\left(\left(\frac{(\log s)^b}{s}\right)^{1/a}\right) &\lesssim& \frac{s}{(\log s)^b}~\left(\log\left(\frac{s^{1/a}}{(\log s)^{b/a}}\right)\right)^b\\
&=& s \left(\frac{\log s^{1/a}-\log\left(\log s\right)^{b/a}}{\log s}\right)^b
\eeas
Since the term inside the bracket on the right-hand side goes to a finite positive limit as $s$ goes to infinity, we get that
\bes
h^{-1}\left(\left(\frac{(\log s)^b}{s}\right)^{1/a}\right)\lesssim s, \:\: \textit{ for } s\geq 1.
\ees
Since $h$ is non-increasing, it follows that
\bes
h(s)\lesssim \frac{1}{s^{1/a}}\left(\log s\right)^{b/a}.
\ees
\end{proof}

\begin{lem} \label{lem-phiintegration}
Let $0< \alpha< 3/2, \zeta>0$ and $0< \beta< n-\alpha$. Then for each $a>0$, we have $\int_{a}^\infty \left([k_{\alpha}\ast k_{\zeta, \beta}]^\ast(t)\right)^2~dt< \infty$.
\end{lem}
\begin{proof}
By Lemma \ref{lem-k-ast}, it is enough to show that there exists $c_0>0$ such that the integral $\int_{c_0}^\infty \left([k_{\alpha}\ast k_{\zeta, \beta}]^\ast(t)\right)^2~dt< \infty$. 
We set
\beas
f_1(x)&=&\frac{1}{|x|^{n-\alpha-\beta}}, \:\: |x|<1;\\
&=& 0, \:\:\: \hspace{1cm} |x|\geq 1,
\eeas
and $f_2(x)=\left(\chi_{1/2} |\cdot|^{\alpha-l-2|\Sigma_0^+|}~\phi_0(\cdot)\right)\ast k_{\zeta, \beta} (x)$. Let $\zeta^\prime\in (0, \zeta)$. By (\ref{conv-est-infty}) there exists $C>0$ such that
\be \label{eqn-conv}
k_{\alpha}\ast k_{\zeta, \beta}(x)\leq C\left(f_1(x)+e^{-\zeta^\prime |x|}~\phi_0(x)+ f_2(x)\right), \:\: x\in X.
\ee
Proceeding as before we get by (\ref{eqn-conv1}) that for $t> |{\bf B}(o, 1)|$ 
\bes 
\left(f_1+e^{-\zeta^\prime |\cdot|}~\phi_0+ f_2\right)^\ast(t) \nonumber
= 2\left(e^{-\zeta^\prime |\cdot|}~\phi_0+ f_2\right)^\ast(t-c_0),
\ees
where $c_0=|{\bf B}(o, 1)|$. By (\ref{eqn-conv}) and equation above it follows that
\beas
\int_{c_0}^\infty \left[(k_{\alpha}\ast k_{\zeta, \beta})^\ast(t)\right]^2~dt &\leq& 4C^2 \int_{c_0}^\infty \left[\left( e^{-\zeta'|\cdot|}~\phi_0 + f_2\right)^\ast(t-c_0)\right]^2~dt\\
&=& 4C^2 \int_X \left|e^{-\zeta'|x|}~\phi_0(x) + f_2(x)\right|^2~dx\\
&\leq& 4C^2 \left\{\left(\int_X e^{-2\zeta'|x|}~\phi_0(x)^2~dx\right)^{1/2} + \left(\int_X \left(f_2(x)\right)^2~dx\right)^{1/2}\right\}^{2}.
\eeas
The integral formula (\ref{int-formula}) and the estimate (\ref{est-phi0}) yield
\bes
\int_{X} e^{-2\zeta'|x|}~\phi_0(x)^2~dx<\infty.
\ees
On the other hand, by Plancherel formula (\ref{plancherel}) for $K$-biinvariant functions we have for $0< \alpha< 3/2$ that
\beas
\int_{X} \left(f_2(x) \right)^2 ~dx&=& c_G \int_{\mathfrak a^\ast} |\widehat{(k_{\zeta, \beta}})(\lambda)|^2~ \big|\left(\chi_{1/2}|\cdot|^{\alpha-l-2|\Sigma_0^+|}~\phi_0\right)^{\widehat{•}}(\lambda)\big|^2~ |{\bf c}(\lambda)|^{-2}~d\lambda\\
&=& c_G \int_{\mathfrak a^\ast} (|\lambda|^2 +\zeta^2)^{-\beta} ~ \big|\left(\chi_{1/2}|\cdot|^{\alpha-l-2|\Sigma_0^+|}~\phi_0\right)^{\widehat{•}}(\lambda)\big|^2~ |{\bf c}(\lambda)|^{-2}~d\lambda\\
&\leq& C \int_{\mathfrak a^\ast}\big|\left(\chi_{1/2}|\cdot|^{\alpha-l-2|\Sigma_0^+|}~\phi_0\right)^{\widehat{•}}(\lambda)\big|^2~ |{\bf c}(\lambda)|^{-2}~d\lambda\\
&=& C \int_{\{x\in X:|x|\geq 1/2\}} |x|^{2\alpha-2l-4|\Sigma_0^+|} ~\left(\phi_0(x)\right)^2~ dx < \infty.
\eeas
This completes the proof.
\end{proof}

\section{Proof of the theorems}

\begin{proof}[Proof of Theorem \ref{thm0}]
Let $u\in W^{\alpha, p}(X)$ and we write $f= (-\Delta-|\rho|^2+\zeta^2)^{\alpha/2}u$. Then clearly $u= f\ast k_{\zeta, \alpha}$ and by the hypothesis $\|f\|_p\leq 1$. Applying O'Neil's lemma \cite[Lemma 1.5]{On} for the rearrangement of convolution, we have for $t>0$
\bes
u^\ast(t)\leq \frac{1}{t} \int_{0}^t f^\ast(s)~ds \int_{0}^t k_{\zeta, \alpha}^\ast(s)~ds+ \int_{t}^\infty f^\ast(s) k_{\zeta, \alpha}^\ast(s)~ds.
\ees
Therefore,
\bea \label{est1}
&& \frac{1}{|E|} \int_E \exp\left(\beta_0(n, \alpha)|u(x)|^{p^\prime}\right)~dx \leq \frac{1}{|E|} \int_{0}^{|E|} \exp\left(\beta_0(n, \alpha)|u^\ast(t)|^{p\prime}\right)~dt \nonumber\\
&\leq& \frac{1}{|E|} \int_{0}^{|E|} \exp\left(\beta_0(n, \alpha)~ \vline~\frac{1}{t}\int_{0}^t f^\ast(s)~ds \int_{0}^t k_{\zeta, \alpha}^\ast(s)~ds+ \int_{t}^\infty f^\ast(s)k_{\zeta, \alpha}^\ast(s)~ds~\vline\:^{p^\prime}\right)~dt \nonumber\\
&\leq& \int_{0}^\infty \exp \bigg(-t+\beta_0(n, \alpha)~\vline~\frac{1}{|E|e^{-t}} \int_{0}^{|E|e^{-t}} f^\ast(s)~ds~\int_{0}^{|E|e^{-t}} k_{\zeta, \alpha}^\ast(s)~ds \nonumber \\
&& \hspace{4cm} + \int_{|E|e^{-t}}^{\infty} f^\ast(s)~k_{\zeta, \alpha}^\ast(s)~ds~ \vline\:^{p^\prime} \bigg)~dt.
\eea
To get the last equation, we use the substitution $t \mapsto |E| e^{-t}$. Next, we change the variables
\bea
&&\phi(t)= (|E|e^{-t})^{1/p} f^\ast(|E|e^{-t}); \label{defn-phi} \\
&& \psi(t)= \beta_0(n, \alpha)^{1/p^\prime} (|E|e^{-t})^{1/p^\prime} ~k_{\zeta, \alpha}^\ast(|E|e^{-t}) \label{defn-psi}.
\eea
It is now easy to check that
\beas
&& \int_{0}^{|E|e^{-t}} f^\ast(s)~ds~\int_{0}^{|E|e^{-t}} k_{\zeta, \alpha}^\ast(s)~ds= \frac{|E|}{\beta_0(n, \alpha)^{1/p^\prime}}~\int_{t}^{\infty} e^{-s/p^\prime}~\phi(s)~ds~\int_{t}^{\infty} e^{-s/p}~\psi(s)~ds;\\
&&  \int_{|E|e^{-t}}^{\infty} f^\ast(s)~k_{\zeta, \alpha}^\ast(s)~ds= \frac{1}{\beta_0(n, \alpha)^{1/p^\prime}}~\int_{-\infty}^t~\phi(s)~\psi(s)~ds.
\eeas
Putting the above quantities in (\ref{est1}) we get that
\bea \label{est2}
&& \frac{1}{|E|} \int_E \exp\left(\beta_0(n, \alpha)|u(x)|^{p^\prime}\right)~dx \nonumber\\
&\leq& \int_{0}^\infty \exp \bigg(-t+ ~\vline~e^t \int_{t}^{\infty} e^{-s/p^\prime}~\phi(s)~ds~\int_{t}^{\infty} e^{-s/p}~\psi(s)~ds  + \int_{-\infty}^t~\phi(s)~\psi(s)~ds~ \vline\:^{p^\prime} \bigg)~dt \nonumber\\
&& =\int_{0}^\infty e^{-F(t)}~dt,
\eea
where 
\bes
F(t)= t-\left( e^t \int_{t}^{\infty} e^{-s/p^\prime}~\phi(s)~ds~\int_{t}^{\infty} e^{-s/p}~\psi(s)~ds  + \int_{-\infty}^t~\phi(s)~\psi(s)~ds~ \right)^{p^\prime}.
\ees
We now set
\bea \label{defn-a}
a(s,t)&=& \psi(s), \:\:\:\: \hspace{4cm} s<t;\\
&=& e^t\left(\int_{t}^\infty e^{-r/p}\psi(r)~dr\right)~e^{-s/p^\prime}, \:\:  s>t \nonumber. 
\eea
Then, by (\ref{est2}) we have
\bes
\frac{1}{|E|} \int_E \exp\left(\beta_0(n, \alpha)|u(x)|^{p^\prime}\right)~dx\leq \int_{0}^\infty e^{-F(t)}~dt,
\ees
where
\be \label{defn-F}
F(t)=t-\left(\int_\R a(s,t)~\phi(s)~ds\right)^{p^\prime}.
\ee
Now, we prove that there exists $C$ independent of $u$ such that $\int_{0}^\infty e^{-F(t)}~dt\leq C$. The proof is inspired by similar ideas used by Adams \cite[Lemma 1]{Ad} and has been carried
out in details in \cite{LLY2}. For the sake of completeness, we sketch the proof. First, notice that
\bes
\int_0^\infty e^{-F(t)}~dt= \int_{\R}|E_\lambda|~e^{-\lambda}~d\lambda,
\ees
where $E_\lambda=\{t\geq 0: F(t)\leq \lambda\}$ and $|E_\lambda|$ is the Lebesgue measure of $E_\lambda$. It is enough to show the following two facts:
\begin{enumerate}
\item[(i)] There exists a constant $c \geq 0$ which is independent of $\phi$ such that $\inf_{t\geq 0} F(t) \geq -c$. 
\item[(ii)] There exist constants $B_1$ and $B_2$ which are both independent of $\phi$ and $\lambda$ such that $|E_\lambda|\leq B_1|\lambda|+ B_2$.
\end{enumerate}
We first prove (i). We set $L(t)= \left(\int_{t}^\infty \phi(s)^p~ds\right)^{1/p}$. By the definition (\ref{defn-phi}) of $\phi$, it follows that
\bes
\int_{-\infty}^{t} \phi(s)^p~ds= \int_{\R} \phi(s)^p~ds- L(t)^p= \|f\|_p^p - L(t)^p\leq 1-L(t)^p.
\ees 
By using the above estimate and H\"older's inequality, it follows from (\ref{defn-F}) that if $t\in E_\lambda$,
\bea \label{est3}
t-\lambda &\leq& \left[\int_\R a(s,t)\phi(s)~ds\right]^{p^\prime} \nonumber\\
&=& \left[\int_{-\infty}^t a(s,t)\phi(s)~ds + \int_{t}^\infty a(s,t)\phi(s)~ds\right]^{p^\prime} \nonumber\\
&\leq& \left[\left(\int_{-\infty}^t a(s,t)^{p^\prime}~ds\right)^{1/p^\prime} \left(1-L(t)^p\right)^{1/p} + \left(\int_t^{\infty} a(s,t)^{p^\prime}~ds\right)^{1/p^\prime} L(t) \right]^{p^\prime} \nonumber\\
&=& \bigg[\left(\int_{-\infty}^t \psi(s)^{p^\prime}~ds\right)^{1/p^\prime} \left(1-L(t)^p\right)^{1/p} \nonumber \\
&& +~ e^{t} \left(\int_{t}^\infty e^{-r/p} \psi(r)~dr\right)~ \left(\int_{t}^\infty e^{-s}~ds\right)^{1/p^\prime} ~L(t)\bigg]^{p^\prime}.
\eea
By the definition (\ref{defn-psi}) of $\psi$, the estimate given in Lemma \ref{lem-beta-ast} and the fact that $p=n/ \alpha$ we have
\be \label{est-psi-zero}
\psi(t)= \beta_0(n, \alpha)^{1/p^\prime} (|E|e^{-t})^{1/p^\prime} ~k_{\zeta, \alpha}^\ast(|E|e^{-t})\leq 1+ \mathcal O\left(e^{-\frac{\tilde\epsilon t}{n}}\right), \:\: \textit{ for all } t>0.
\ee
Let $\zeta> 0$ if $1 < p < 2$ and $\zeta> 2|\rho|~(\frac{1}{2}-\frac{1}{p})$ if $p\geq 2$. We choose $\zeta^\prime\in (0, \zeta)$ with $\zeta-\zeta^\prime$ small enough such that $\zeta^\prime$ satisfies the same properties as $\zeta$, that is, $\zeta^\prime> 0$ if $1 < p < 2$ and $\zeta^\prime > 2|\rho|~(\frac{1}{2}-\frac{1}{p})$ if $p\geq 2$. Then by Lemma \ref{lem-est-k-ast-inft} we have
\bes
\int_{-\infty}^0 \psi(s)^{p^\prime}~ds= \beta_0(n, \alpha) \int_{|E|}^\infty \left(k_{\zeta, \alpha}^\ast(t)\right)^{p^\prime}~dt \lesssim \int_{|E|}^\infty \left(t^{-1/2-\zeta^\prime/2|\rho|}~ (\log t)^{2|\rho|l/(\zeta'+|\rho|)}\right)^{p^\prime}~dt<\infty.
\ees
Therefore, using (\ref{est-psi-zero}) and the above estimate we have
\be \label{est4}
\int_{-\infty}^t \psi(s)^{p^\prime}~ds= \int_{-\infty}^0 \psi(s)^{p^\prime}~ds + \int_{0}^t \psi(s)^{p^\prime}~ds\leq b_1+ \int_{0}^t \left(1+ \mathcal O(e^{-\tilde\epsilon s/n})\right)^{p^\prime}~ds\leq b_2+t, 
\ee
and
\bea \label{est5}
e^{t} \left(\int_{t}^\infty e^{-r/p} \psi(r)~dr\right)~\left(\int_{t}^\infty e^{-s}~ds\right)^{1/p^\prime} &\leq& e^t \int_{t}^\infty e^{-r/p} \left(1+\mathcal O(e^{-\tilde\epsilon r/n})\right)~dr~ e^{-t/p^\prime} \nonumber\\
&\leq & C \int_{t}^\infty e^{-(r-t)/p}~dr=b_3< \infty,
\eea
where the constants $b_1 , b_2$ and $b_3$ are independent of $\phi$. Using the estimates (\ref{est4}), (\ref{est5}) it follows from (\ref{est3}) that
\bes
t-\lambda \leq \left[\left(b_2+t\right)^{1/p^\prime}\left(1-L(t)^p\right)^{1/p} +b_3 L(t)\right]^{p^\prime}.
\ees
The rest of the proof of (a) is similar to that in \cite{Ad} (see the proof after eqn.(16) in \cite{Ad}). Since the proof of (b) is the same as that in \cite{LLY2} we will omit here.

The sharpness of the constant $\beta_0(n, \alpha)$ can be verified by the process similar to that in \cite{Ad, KSW, LLY2, RuS} and thus the proof of Theorem \ref{thm0} is completed.
\end{proof}

\begin{proof}[Proof of Theorem \ref{thm1}]
Let $u\in W^{\alpha, p}$ with $\int_X |(-\Delta-|\rho|^2+\zeta^2)^{\alpha/2} u(x)|^p~dx\leq 1$. By Corollary \ref{cor-sobolev}, we have
\bes
\int_{X}|u(x)|^p~dx\leq S_p\int_X |(-\Delta-|\rho|^2+\zeta^2)^{\alpha/2} u(x)|^p~dx\leq S_p,
\ees
provided $\zeta> 2|\rho||1/2-1/p|$. We now set $\Omega(u)=\{x\in X: |u(x)|\geq 1\}$. Then from the above inequality it follows that
\bes
|\Omega(u)|=\int_{\Omega(u)}~dx\leq \int_{X}|u(x)|^p~dx\leq S_p.
\ees
Therefore, we have $|\Omega(u)|\leq S_p$,  which is independent of $u$ satisfying $\|(-\Delta-|\rho|^2+\zeta^2)^{\alpha/2}u\|_p\leq 1$ provided $\zeta> 2|\rho||1/2-1/p|$. We now write
\bes
\int_X \Phi_p\left(\beta_0(n, \alpha)|u(x)|^{p^\prime}\right)~dx= \int_{\Omega(u)} \Phi_p\left(\beta_0(n, \alpha)|u(x)|^{p^\prime}\right)~dx + \int_{X\backslash \Omega(u)} \Phi_p\left(\beta_0(n, \alpha)|u(x)|^{p^\prime}\right)~dx.
\ees
We now notice that $j_p=p$ if $p$ is an integer and $j_p=[p]+1$ if $p$ is not an integer. Therefore, $(j_p-1)p^\prime\geq p$ for all $p>1$. We also notice that on the domain $X\backslash \Omega(u)$,  $|u(x)|< 1$. Thus
\bea \label{est-s}
\int_{X\backslash \Omega(u)} \Phi_p\left(\beta_0(n, \alpha)|u(x)|^{p^\prime}\right)~dx &\leq& \sum_{k=j_p-1}^\infty \frac{\beta_0(n, \alpha)^k}{k!} \int_{X\backslash \Omega(u)} |u(x)|^{p^\prime k}~dx \nonumber\\
&\leq& \sum_{k=j_p-1}^\infty \frac{\beta_0(n, \alpha)^k}{k!} \int_{X\backslash \Omega(u)} |u(x)|^p~dx \nonumber\\
&\leq& \sum_{k=j_p-1}^\infty \frac{\beta_0(n, \alpha)^k}{k!} \|u\|_p^p\leq C_1.
\eea
Since $\zeta> 2|\rho||1/p-1/2|$, by Theorem \ref{thm0} there exists $C_2>0$ independent of $u$ such that
\be \label{est-omega}
\int_{\Omega(u)} \Phi_p\left(\beta_0(n, \alpha)|u(x)|^{p^\prime}\right)~dx\leq \int_{\Omega(u)} \exp\left(\beta_0(n, \alpha)|u(x)|^{p^\prime}\right)~dx\leq C_2.
\ee
Combining equations (\ref{est-s}) and (\ref{est-omega}) it follows that
\bes
\int_{X} \Phi_p\left(\beta_0(n, \alpha)|u(x)|^{p^\prime}\right)~dx\leq C_1+C_2= C,
\ees
for all $u$ satisfying $\|(-\Delta-|\rho|^2+\zeta^2)^{\alpha/2}\|_p\leq 1$ provided $\zeta> 2|\rho||1/2-1/p|$.
The sharpness of the constant $\beta_0(n, \alpha)$ can be verified by the process similar to that in the proof of Theorem \ref{thm0}.
\end{proof}

\begin{lem} \label{lem-ps2}
Let $n\geq 3$, $\zeta>0$ and $0<2s< \min\{l+2|\Sigma_0^+|, n\}$. Then for $2< q\leq \frac{2n}{n-2s}$ there exists $C=C(n, s, q,\zeta)$ such that for all $u\in W^{n/2, 2}(X)$
\bes
\int_X |(-\Delta -|\rho|^2)^{s/2}(-\Delta-|\rho|^2 +\zeta^2)^{(n-2s)/4}u(x)|^2~dx\geq C\|u\|_q^2.
\ees
\end{lem}
\begin{proof}
By the Plancherel formula (\ref{plancherel}) and the Poincar\'e-Sobolev inequality (Theorem \ref{thm-p-s}), it follows that
\beas
&& \int_X|(-\Delta -|\rho|^2)^{s/2}(-\Delta-|\rho|^2+\zeta^2)^{(n-2s)/4}u(x)|^2~dx\\
&=& c_G \int_{\mathfrak a_+^\ast \times K} |\lambda|^{2s}~(|\lambda|^2 +\zeta^2)^{(n-2s)/2} ~|\widetilde u(\lambda, k)|^2~|{\bf c}(\lambda)|^{-2}~d\lambda~dk\\
&\geq& \zeta^{n-2s} c_G \int_{\mathfrak a_+^\ast \times K} |\lambda|^{2s} ~|\widetilde u(\lambda, k)|^2~|{\bf c}(\lambda)|^{-2}~d\lambda~dk\\
&=& \zeta^{n-2s} \int_X |(-\Delta-|\rho|^2)^{s/2}u(x)|^2~dx \geq C \|u\|_q^2 ,\:\:\:\: 2< q\leq\frac{2n}{n-2s}.
\eeas
This completes the proof.
\end{proof}

\begin{proof}[Proof of Theorem \ref{thm3}]
Let $u\in W^{n/2, 2}(X)$ satisfying (\ref{hypo-thm3}). We choose some $q_0$ satisfying $2<q_0\leq \min\{2n/(n-2s), 4\}$. Then by Lemma \ref{lem-ps2} we have
\bes
\|u\|_{q_0}^2\leq C_0 \int_X|(-\Delta-|\rho|^2)^{s/2}(-\Delta-|\rho|^2+\zeta^2)^{(n-2s)/4}u(x)|^2~dx\leq C.
\ees
We now set $\Omega(u)= \{x\in S:|u(x)|\geq 1\}$, then 
\bes
|\Omega(u)|= \int_{\Omega(u)} dx\leq \int_X |u(x)|^{q_0}~dx\leq C^{q_0/2},
\ees
where the constant $C^{q_0/2}$ is independent of $u$. Since $q_0\leq 4$, it follows that
\bea \label{eqn-thm11}
&& \int_{X\backslash \Omega(u)} \left[\exp\left(\beta_0(n, n/2)u(x)^2\right)-1- \beta_0(n, n/2)u(x)^2\right]~dx \nonumber\\
&=& \sum_{k=2}^{\infty} \frac{\beta_0(n, n/2)^k}{k!} \int_{X\backslash \Omega(u)} u(x)^{2k}~dx \nonumber\\
&\leq & \sum_{k=2}^\infty \frac{\beta_0(n, n/2)^k}{k!} \int_{X\backslash \Omega(u)} |u(x)|^{q_0}~dx< \infty.
\eea
Next, we show that $\int_{\Omega(u)} \exp\left(\beta_0(n, n/2)u(x)^2\right)~dx$ is bounded by some constant independent of $u$. We rewrite 
\bes
v= (-\Delta-|\rho|^2)^{s/2}(-\Delta-|\rho|^2+\zeta^2)^{(n-2s)/4}~u.
\ees
Then $\|v\|_2\leq 1$ and $u= v\ast (k_{s} \ast k_{\zeta, (n-2s)/2})$. By Lemma \ref{lem-k-ast}, the kernel $k_{s} \ast k_{\zeta, (n-2s)/2}$ satisfies
\bes
[k_{s} \ast k_{\zeta, (n-2s)/2}]^\ast(t) \leq \frac{1}{\gamma(n/2)} \cdot \left(\frac{nt}{\omega_{n-1}}\right)^{-1/2} + O(t^{-1/2+\epsilon'/n}), \:\: \txt{ for } 0<t<1, 
\ees
and by Lemma \ref{lem-phiintegration}, for each $a>0$
\bes
\int_{a}^\infty |[k_{s} \ast k_{\zeta, (n-2s)/s}]^\ast(t)|^2~dt < \infty.
\ees
Following the proof of Theorem \ref{thm0}, we can find a constant $C$ independent of $u$ such that 
\bes
\int_{\Omega(u)} \exp\left(\beta_0(n, n/2)u(x)^2\right)~dx= \int_{\Omega(u)} \exp \left(\beta_0(n, n/2)[v\ast (k_{s} \ast k_{s,(n-2s)/2})]^2\right)~dx \leq C.
\ees
Combining equation (\ref{eqn-thm11}) with the above inequality we complete the proof. 

The sharpness of the constant $\beta_0(n, \alpha)$ can be verified by the process similar to that in the proof of Theorem \ref{thm0}
\end{proof}

\begin{proof}[Proof of Theorem \ref{thm-HA}]
Let $u\in C_c^\infty(X)$ with 
\bes
\int_X |(-\Delta)^{s/2}(-\Delta-|\rho|^2+\zeta^2)^{(n-2s)/4} u(x)|^2~dx- |\rho|^{2s}\zeta^{n-2s}\int_{X}|u(x)|^2~dx\leq 1.
\ees
The Plancherel formula (\ref{plancherel}) yields 
\beas
&&\int_X |\left(-\Delta- |\rho|^2\right)^{s/2} \left(-\Delta-|\rho|^2+\zeta^2\right)^{(n-2s)/4} u(x)|^2~dx\\
&=& c_G \int_{\mathfrak a_+^\ast \times K} |\lambda|^{2s}(|\lambda|^2+\zeta^2)^{(n-2s)/2}~|\widetilde u(\lambda, k)|^2~|{\bf c}(\lambda)|^{-2}~d\lambda~dk\\
&\leq & c_G \int_{\mathfrak a_+^\ast\times K} \left[\left(|\lambda|^2+|\rho|^2\right)^s \left(|\lambda|^2+\zeta^2\right)^{(n-2s)/2}- \zeta^{n-2s}|\rho|^{2s}\right]~|\widetilde u(\lambda, k)|^2~|{\bf c}(\lambda)|^{-2}~d\lambda~dk\\
&=& \int_X |(-\Delta)^{s/2}(-\Delta-|\rho|^2+\zeta^2)^{(n-2s)/4}u(x)|^2~dx-\zeta^{n-2s} |\rho|^{2s}\int_{X} |u(x)|^2~dx\leq 1.
\eeas
Therefore, by Theorem \ref{thm3} we complete the proof.
\end{proof}

We conclude the paper with the following remark.
\begin{rem}\label{finalrem}
We recall that Damek-Ricci spaces are non-symmetric generalization of
rank one Riemannian symmetric spaces. Though symmetric spaces are the most important prototypes, they form a very small subclass of the set of all Damek-Ricci spaces (see \cite{ADY}). A Damek-Ricci space is a Riemannian manifold and a solvable Lie group but in general not a symmetric space, i.e. cannot be realized as a quotient space $G/K$, for
a semisimple Lie group $G$. It will be interesting to see whether Adams-type inequalities can be proved in the context of Damek-Ricci spaces. 
\end{rem}

\noindent{\bf Acknowledgement:} The author is supported by INSPIRE Faculty Award (IFA19-MA136) from Department of Science and Technology, India.

\end{document}